\documentclass[a4paper, 11pt, reqno]{amsart}
\usepackage{graphicx}
\usepackage{amsmath}
\usepackage{amsfonts}
\usepackage{amssymb}
\usepackage{xcolor}
\usepackage{pdfpages}
\usepackage{comment}
\usepackage{enumitem}
\usepackage{quiver}
\usepackage{amsthm}
\usepackage{extarrows}
\usepackage{mathtools}
\usepackage{graphicx}
\usepackage{hyperref}
\usepackage{cleveref}
\usepackage[backend=biber]{biblatex}
\bibstyle{authoryear}
\usepackage{bm}
\crefname{subsection}{subsection}{subsections}
\Crefname{subsection}{Subsection}{Subsections}
\newtheorem{thm}{Theorem}[section]
\newtheorem{definition}[thm]{Definition}
\newtheorem{lem}[thm]{Lemma}

\newtheorem{cor}[thm]{Corollary}
\newtheorem {remark}[thm]{Remark}
\newtheorem*{claim}{Claim}
\theoremstyle{definition}
\newcommand{\mt}[2]{\begin{pmatrix} #1 \\ #2 \end{pmatrix}}
\newcommand{\matr}[4]{\begin{pmatrix} #1 & #2 \\ #3 & #4 \end{pmatrix}}
\newcommand{\R}{\mathbb{R}}
\newcommand{\Z}{\mathbb{Z}}
\newcommand{\T}{\mathbb{T}}
\newcommand{\C}{\mathbb{C}}
\newcommand{\N}{\mathbb{N}}
\newcommand{\D}{\mathcal{D}}
\newcommand{\Ll}{\mathcal{L}}
\newcommand{\Cc}{\mathcal{C}}
\newcommand{\lemref}[1]{Lemma~\ref{#1}}
\newcommand{\thmref}[1]{Theorem~\ref{#1}}
\newcommand{\corref}[1]{Corollary~\ref{#1}}

\DeclarePairedDelimiter{\ceil}{\lceil}{\rceil}

\title[Extreme values for Toral translations]{Extreme values and logarithm laws for toral translations}
\author{Sourav Das}
\address{\textbf{Sourav Das} \\
  Department of Mathematics, Ben-Gurion University of the Negev, Beer Sheva, Israel 8410501}
\email{dassou@post.bgu.ac.il, iamsouravdas1@gmail.com}

\author{Anish Ghosh} 
\address{\textbf{Anish Ghosh} \\
School of Mathematics,
Tata Institute of Fundamental Research, Mumbai, India 400005}
\email{ghosh@math.tifr.res.in}

\author{Sundara Narasimhan} 
\address{\textbf{Sundara Narasimhan} \\
School of Mathematics,
Tata Institute of Fundamental Research, Mumbai, India 400005}
\email{ssundara@math.tifr.res.in, sundaranarasimhan.s@gmail.com}

\thanks{All three authors acknowledge support from the Department of Atomic Energy, Government of India (project 12-R\&D-TFR-5.01-0500). A.\ G.\ gratefully acknowledges support from a grant from the Infosys Foundation, and a J. C. Bose grant of the ANRF}

\begin{document}
\maketitle

\begin{abstract}
We study extreme values of normalised $r$-th nearest neighbour distance functions in inhomogeneous Diophantine approximation, and prove Weibull and Fr\'echet limit laws for these functions when the shift satisfies a natural Diophantine condition. Using these extreme value laws we prove corresponding logarithm laws for these functions.

Our approach is based on homogeneous dynamics. We adapt the Poisson approximation strategy of Bj$\ddot{\text{o}}$rklund and Gorodnik \cite{bjorklund2023poisson} combined with an effective multiple equidistribution result, which we deduce using results of  Str$\ddot{\text{o}}$mbergsson \cite{strom} and of Bj\"{o}rklund and Gorodnik \cite{BMGA}. We also need the volume estimates of Str\"{o}mbergsson and Venkatesh \cite{stromsmall} as crucial input.
\end{abstract}
\tableofcontents

\section{Introduction}
Let $\T$ denote the torus $\R/\Z$ and let $d(\cdot,\cdot)$ be the distance on $\T$ defined as:$$d(x,y):=\min_{p\in\Z} \{|x-y+p|\}.$$For $T \geq 1$ and $\alpha \in \T$, let $k(T;\alpha)$ denote the function given by $$k(T;\alpha):=T\min\{d(q\alpha, 0): \; 1\leq q\leq T, \; q \in \N\}.$$ Dirichlet's classical theorem in Diophantine approximation gives us that  $k(T;\alpha)\leq 1$ for all $T\geq 1$ and for all $\alpha \in \T$. A result of Davenport and Schmidt \cite{davenport1970dirichlet} strengthens Dirichlet by showing that for almost all $\alpha \in \T$ one has \begin{equation}\label{Davenport-Schmidt}
    \limsup_{T \to \infty} k(T;\alpha)=1.
\end{equation}

 At the same time, a simple consequence of Khintchine's theorem tells us that for almost all $\alpha \in \T$ the inequality  \begin{equation}\label{Khincthine}
    k(T;\alpha)\leq\frac{1}{\log(T)}
\end{equation} holds for infinitely many $T$. Thus, from \eqref{Davenport-Schmidt} and \eqref{Khincthine} one concludes that the function $k(T;\alpha)$ has big fluctuations.

One is therefore naturally led to the study of the distribution of extreme values of the function $k(T;\alpha)$. This was carried out by Bj\"{o}rklund and Gorodnik in dimension $\geq 2$ (see \cite[Section 1.3]{bjorklund2023poisson} for more details). They also derived ``logarithm laws" from their extreme value results.

We now turn to the problem of inhomogeneous Diophantine approximation, i.e., approximating a given $\xi \in \T$ with elements of the form $\{q\alpha\}_{q=1}^\infty$. Given $\alpha, \xi 
\in \T$ with $\xi \neq 0$ and $T \geq 1,$ we define the function $$k_{\xi}(T;\alpha):=T\min\{d(q\alpha, \xi): 0\leq q\leq T;\; q\in \Z_{\geq 0}\}.$$ 
In contrast to the homogeneous setting (i.e. $\xi=0$), an analogue of Dirichlet's theorem fails to hold for inhomogeneous Diophantine approximation (see \cite[Chapter 3, Theorem 3]{Cassels1957}). Also, the inhomogeneous analogue of Davenport-Schmidt \eqref{Davenport-Schmidt} is strikingly different. By a consequence of \cite[Theorem 1.6]{KWa} (or by an ergodicity argument), we have for almost every $(\alpha, \xi) \in \T^2$ \begin{equation}
    \limsup_{T \to \infty} k_{\xi}(T;\alpha)=\infty.
\end{equation}

However, at the same time, a similar application of the inhomogeneous version of Khintchine's theorem \cite[Satz 2]{Szusz1958} gives us that for a fixed $\xi \in \T$, one has for almost every $\alpha \in \T$
\begin{equation}
    k_{\xi}(T;\alpha)\leq \frac{1}{\log(T)}
\end{equation}
for infinitely many $T$. This indicates that, like $k(T;\alpha)$, the function $k_\xi(T;\alpha)$ also has large fluctuations naturally leading one to investigate the extreme values of $k_\xi(T;\alpha)$ and corresponding logarithm laws, where $\xi$ is fixed. This is the content of our article.  To state our results we make the following definitions.
\begin{definition}
Given $\alpha, \xi \in \mathbb{T}$ with $\xi \neq 0$, $r \in \N$, and $T \in \mathbb{R}_{\geq r}$, we set  
\begin{equation}
    d^{(r)}_\xi(T;\alpha):= r\text{-th minimum among }  \{d(q\alpha, \xi) : \; 0 \leq q \leq T,\,\, q \in \Z_{\geq 0} \}
\end{equation}
and
\begin{equation}k^{(r)}_{\xi}(T;\alpha):= T \cdot d^{(r)}_\xi(T;\alpha).\end{equation}
Clearly $$k_\xi(T;\alpha)=k^{(1)}_{\xi}(T;\alpha)\leq k^{(2)}_{\xi}(T;\alpha)\leq k^{(3)}_{\xi}(T;\alpha)\ldots .$$

For $\xi=0$, take the minima over the set $1\leq q \leq T$ instead of $0\leq 
q \leq T$ for the definition of $d^{(r)}_0(T;\alpha)$ and $k_{0}^{(r)}(T; \alpha)$.
\end{definition}

These functions have been earlier studied by Dolgopyat, Fayad, and Liu in \cite[Section 4]{Dmitry} on the $d$-dimensional torus $\T^d$. In \cite[Theorem 4.7]{Dmitry}, they proved logarithm laws for these functions $k_{\xi}^{(r)}(T;\alpha)$ for almost every pair $(\alpha,\xi) \in \T^2$. Inspired by \cite{Dmitry} and \cite{bjorklund2023poisson}, we study the extreme values of these functions and prove corresponding logarithm laws. We prove that for sufficiently lacunary subsets $\Delta_n$ of $\R_{\geq 1}$ the (appropriately scaled) minima and maxima of the functions $k^{(r)}_\xi(T;\alpha)$ have a limiting distribution. To state the main results of this article we first need the following definitions:
\begin{definition}
    We say that $\xi \in \mathbb T$ is of Diophantine type if there exists constants $M \geq 2$ and $c>0$ such that
    $$\left|\xi-\frac{p}{q}\right|\geq \frac{c}{q^M} \quad \text{for all } p\in \mathbb Z \text{ and } q \in \mathbb N.$$
\end{definition}

Note that $\xi$ being of Diophantine type is a full measure condition, and the complement is a set of Hausdorff dimension zero. Also, the set of numbers of Diophantine type includes all algebraic irrationals $\pmod{1}$ by Liouville's theorem.

We also recall the definitions of Weibull and Fr\'echet distributions.
\begin{definition}
The Weibull distribution with scale parameter $(\lambda>0)$ and shape parameter $(a>0)$, denoted by $\mathrm{Wei}(\lambda,a)$, is the probability distribution whose cumulative distribution function is given by
\[
F(t)=\begin{cases}
     1-e^{-\lambda t^a} & \text{if} \quad t>0 \\
     0 & \text{if} \quad t\leq 0.
     \end{cases} 
\]
\end{definition}

\begin{definition}
    The Fr\'echet distribution with scale parameter ($\lambda>0$) and shape parameter $(a >0 )$, denoted by $\mathrm{Fre}(\lambda, a)$, is the probability distribution whose cumulative distribution function is given by 

    \[
F(t)=\begin{cases}
     e^{-\frac{\lambda}{t^a}} & \text{if} \quad t>0 \\
     0 & \text{if} \quad t\leq0.
     \end{cases} 
\]

\end{definition}
The following are the main results of the paper.

\begin{thm}\label{thm:wei}
    Fix $\xi \in \mathbb{T}$ of Diophantine type. Let $\Delta_n$ be a sequence of finite subsets of $[1, \infty)$ such that 
\begin{equation*}
    |\Delta_n| \to \infty, \frac{{\min}_{T \neq T^{'} \in \Delta_n}\{|\log T -\log T'|\}}{\log|\Delta_n|} \to \infty, \frac{{\min}_{T \in \Delta_n}\{\log T\}}{\log|\Delta_n|} \to \infty \,\, \text{ as } n \to \infty.
\end{equation*}
Then the minima functions $$\mathfrak{m}^{(r)}_{\xi}(\Delta_n;\alpha):=\min_{T 
\in \Delta_n}k^{(r)}_\xi(T;\alpha), $$
satisfy 
\begin{equation*}
   |\Delta_n|\mathfrak{m}^{(1)}_\xi(\Delta_n; \cdot)\xLongrightarrow[\textnormal{Leb}]{} \textnormal{Wei}(2, 1),
\end{equation*}
\begin{equation*}
    |\Delta_n|^{\frac{1}{2}}\mathfrak{m}^{(r)}_{\xi}(\Delta_n; \cdot) \xLongrightarrow[\textnormal{Leb}]{} \textnormal{Wei}\bigg(\frac{12}{\pi^2}\bigg(\frac{1}{(r-1)^2}-\frac{1}{r^2}\bigg), 2\bigg) \,\, \text{for} \; r\geq2.
\end{equation*}
Also, the maxima function
$$\mathfrak{M}^{(1)}_{\xi}(\Delta_n;\alpha):=\max_{T 
\in \Delta_n}k^{(1)}_\xi(T;\alpha), $$
satisfies
\begin{equation*}
   |\Delta_n|^{-1}\mathfrak{M}^{(1)}_\xi(\Delta_n; \cdot)\xLongrightarrow[\textnormal{Leb}]{} \textnormal{Fre}\left(\frac{1}{\pi^2}, 1\right).
\end{equation*}

\end{thm}

An analogous result can also be derived in the case $\xi=0$. 
\begin{thm} \label{homogeneous}
    Let $\Delta_n$ be a sequence of finite subsets of $[1, \infty)$ such that 
\begin{equation*}
    |\Delta_n| \to \infty, \frac{{\min}_{T \neq T^{'} \in \Delta_n}\{|\log T -\log T'|\}}{\log|\Delta_n|} \to \infty, \frac{{\min}_{T \in \Delta_n}\{\log T\}}{\log|\Delta_n|} \to \infty \,\, \text{ as } n \to \infty.
\end{equation*}
Then the minima functions $$\mathfrak{m}^{(r)}(\Delta_n;\alpha):=\min_{T 
\in \Delta_n}k^{(r)}_0(T;\alpha), $$
satisfy 
\begin{equation*}
   |\Delta_n|\mathfrak{m}^{(r)}(\Delta_n; \cdot)\xLongrightarrow[\textnormal{Leb}]{} \textnormal{Wei}\left(\frac{12}{\pi^2r^2}, 1\right) \; \text{for all} \; r\geq 1. 
\end{equation*}
Also, the maxima function 
$$\mathfrak{M}^{(1)}(\Delta_n;\alpha):=\max_{T 
\in \Delta_n}k^{(1)}_0(T;\alpha),$$
satisfies 
\begin{equation*}
   |\Delta_n|^{1/2}\log(|\Delta_n|)^{1/2}(1-\mathfrak{M}^{(1)}(\Delta_n; \cdot))\xLongrightarrow[\textnormal{Leb}]{} \textnormal{Wei}\left(\frac{3}{\pi^2}, 2\right).
\end{equation*}
\end{thm}

A few remarks are in order:

\medskip
\noindent\textbf{Remarks.}
\begin{enumerate}

\item The lacunarity assumption on the sets $\Delta_n \subset [1, \infty)$ in the above theorems is same as the one made in \cite{bjorklund2023poisson}. This assumption ensures that one has ``asymptotic $r$-independence" for all $r\geq 2$ when sampled along the sequence $\Delta_n$. This heuristic is captured in the form of a Poisson approximation theorem for shrinking targets (see \thmref{Poi law}). As illustrated in \cite{bjorklund2023poisson}, a result of the above kind implies logarithm laws for the associated functions. We have stated them below as corollaries.

\item The different normalisations and the parameters occurring in the limiting distributions of \thmref{thm:wei} and \thmref{homogeneous} are dictated by small-volume asymptotics of the corresponding shrinking targets. See sections \ref{sec:vol}, \ref{sec:wei} and \ref{sec:homo} for more details on how the shrinking volume rates affect the limit laws.

\item One could ask for an analogue of the above theorems for almost every pair $(\alpha, \xi) \in \T^2$. Such a result is true with exactly the same limiting laws as in the above theorems, with the same constants. The reason is that the dynamics one encounters while allowing both $\alpha$ and $\xi$ to vary involves the full two-dimensional unstable horospherical orbit, whereby one can directly appeal to effective multiple equidistribution of translates imitating our strategy; we leave the details to the interested reader.
\end{enumerate}

As corollaries we obtain logarithm laws in the spirit of Sullivan's famous paper \cite{sullivan1982disjoint}. In fact we obtain ``higher" analogues inspired by the work of Dolgopyat, Fayad and Liu \cite{Dmitry}. In \cite{Dmitry}, the authors systematically develop a multiple version of the standard Borel-Cantelli lemma and give several applications in dynamics using quantitative multiple mixing. For example, they prove logarithm laws (and extreme value theorems) for $r$-th deepest cusp excursions of the geodesic flow on non-compact finite volume hyperbolic manifolds extending Sullivan (see \cite[Section 7]{Dmitry} for more details). In the same paper they also study the functions $k^{(r)}_\xi(T;\alpha)$ and prove logarithm laws for them for almost every $(\alpha, \xi) \in \T^2$. See the discussion in \cite[Section 9.11]{Dmitry} and the references mentioned therein for similar results when fixing $\alpha$ and varying $\xi$. The corollaries below can be viewed as the fixed $\xi$ case of their logarithm laws \cite[Theorem 4.7]{Dmitry}:
\begin{cor}\label{cor1}
    Fix $\xi \in \T$ of Diophantine type, then for almost every $\alpha \in \T$ we have 
    \begin{equation}\label{cor11}
        \limsup_{T \to \infty}{\frac{-\log(k^{(1)}_{\xi}(T;\alpha))}{\log(\log(T))}}=1
    \end{equation}
    and 
    \begin{equation}\label{cor12}
        \limsup_{T \to \infty}{\frac{-\log(k^{(r)}_{\xi}(T;\alpha))}{\log(\log(T))}}=\frac{1}{2} \quad\text{for all } \; r\geq2.
    \end{equation}
    Moreover, one has for any $\xi \in \T$ of Diophantine type, for almost every $\alpha \in \T$
    \begin{equation}\label{cor13}
        \limsup_{T \to \infty} \frac{\log(k^{(1)}_\xi(T;\alpha))}{\log(\log(T))}=1.
    \end{equation}
\end{cor}

\begin{cor}\label{cor2}
    For almost every $\alpha \in \T$ we have 
    \begin{equation}\label{cor21}
        \limsup_{T \to \infty}{\frac{-\log(k^{(r)}_{0}(T;\alpha))}{\log(\log(T))}}=1 \quad \text{for all} \; r\geq 1.
    \end{equation} Moreover, one has for almost every $\alpha \in \T$
    \begin{equation}\label{cor22}
        \limsup_{T \to \infty}{\frac{-\log(1-k^{(1)}_{0}(T;\alpha))}{\log(\log(T))}} \geq\frac{1}{2}.
    \end{equation}
\end{cor}
A few remarks are in order:

\medskip
\noindent\textbf{Remarks.}
\begin{enumerate}
\item Traditionally, one obtains logarithm laws using (effective) mixing and a Borel-Cantelli argument. It is not immediately clear how to run this argument to derive \corref{cor1} when we are fixing $\xi$. Here the strategy of establishing an extreme value law and deriving the logarithm laws from it is useful and clarifying, even though it comes at the cost of requiring stronger volume asymptotics and multiple effective equidistribution. This is the approach taken in \cref{loglawsviaevl} to prove \corref{cor1} and \corref{cor2}. 

Nevertheless, one can also adapt the traditional approach to derive logarithm laws in the fixed $\xi$ case. In \cref{loglawwithoutevl}, we illustrate a hands-on method to derive \eqref{cor11} and \eqref{cor12} just using effective double equidistribution (the analogue of mixing). In the direct approach of \cref{loglawwithoutevl}, we show by a careful argument that there exists a $\delta>0$ such that for any subinterval $[\alpha,\beta]\subset [0,1]$, the normalised Lebesgue measure of the set of elements in $[\alpha,\beta]$ satisfying the logarithm law is $>\delta$. This immediately gives that the logarithm law holds almost everywhere by an application of Lebesgue's density theorem. See \cref{loglawwithoutevl} for more details; it can be read independently of the extreme value arguments.

\item \corref{cor1} extends the logarithm law of   Dolgopyat, Fayad, and Liu \cite[Theorem 4.7(a)]{Dmitry} in dimension one by fixing the shift parameter $\xi$. 

\item The first part of \corref{cor2}, i.e. \eqref{cor21}, is already established in \cite[Theorem 4.7 (b)]{Dmitry} in greater generality (i.e., proved in every dimension  $d$). We rederive this result here using \thmref{homogeneous} in \cref{loglawsviaevl}. 

\item The limit in \eqref{cor22} is actually known to be 1 by \cite[Theorem 1.8]{kleinbock2018zero}. See Remark \ref{truelimit} for more details.

\item Note that the distributional laws in \thmref{thm:wei} and \thmref{homogeneous} depend explicitly on $r$, whereas the corresponding logarithm laws in \corref{cor1} and \corref{cor2} are independent of $r$ for all $r \geq 2$.
\end{enumerate}

Extreme value laws arising from homogeneous dynamics have been previously investigated by Kirsebom and Mallahi-Karai \cite{KirsebomMallahi2026} for cusp excursions of horocycle flows on the modular surface, by Marklof and Pollicott \cite{MarklofPollicott2025} for horocycle flows on finite-volume hyperbolic surfaces. For the geodesic flow on the modular surface, see Pollicott \cite{pollicott_limiting_2009}, and Kim, Lee and Lim \cite{kim2026extremevaluetheoremgeodesic} for quotients of the upper half-plane by theta groups.

Turning to Diophantine approximation, the classical works of Doeblin \cite{doeblin1940remarques}, Iosifescu \cite{iosifescu1977poisson} and Galambos \cite{galambos1972distribution} establish limiting distributional laws for the maxima of the coefficients of the continued fraction expansion by showing that exceedances have Poisson asymptotics; that is, for all $u>0$,  $$\mathcal{E}^u_n:=\#\{1\leq i \leq n:\: A_i\log 2>nu \} \xLongrightarrow[\textnormal{Gauss}]{} \textnormal{Poi}(u^{-1}).$$ Here $A_i$ is the $i$-th regular continued fraction coefficient, and the convergence is with respect to the Gauss measure on [0, 1]. More recently, the work of Ghosh, Kirsebom, and Roy \cite{GKR2021} uses the Chen--Stein method to get quantitative Poisson approximation and a rate in Galambos' result. See also the work of Kirsebom \cite{Kirsebom_2021}. Yet another manifestation of Poissonian statistics in Diophantine approximation and dynamics was recently obtained by Álvarez, Becher, Cesaratto, Mereb, Peres, and Weiss \cite{PeresWeiss2026}. They prove that almost every expansion in any numeration system carrying an invariant exponentially ($\psi$)-mixing measure is Poisson generic; in particular, this applies to regular continued fractions. There are also related quantitative and extreme value results in Diophantine approximation on spheres, see \cite{AlamGhosh2022}, \cite{ ouaggag2026contributions},\cite{ gorodnik2026extreme}.

The recent paper \cite{DasGhosh2026} surveys progress in inhomogeneous Diophantine approximation involving quadratic forms. The present results are complementary to the central limit theorem of Aggarwal, Das, and Ghosh \cite{AggarwalDasGhosh2026} for inhomogeneous Diophantine approximation with a fixed shift of Diophantine type (see also \cite{AG25}, \cite{AlamGhoshYu}). Both \cite{AggarwalDasGhosh2026} and this work use an effective multiple equidistribution result of the kind stated in \thmref{multiple}, albeit the theorems and their scope in the two papers are different.  

\subsection*{Acknowledgements}
We thank Jens Marklof and Andreas Str\"{o}mbergsson for useful discussions. 

\subsection*{AI declaration} This paper was conceived, discussed and written entirely by the three human authors without any AI interaction or input. After the first version was written completely, ChatGPT 5.6 Sol was used for proofreading the document. It found an error in the earlier version of \corref{cor2}. This was corrected again entirely by the authors. 

\section{Notation and setup}
\subsection{Notation}
Let $G:=\mathrm{ASL}_2(\mathbb{R}):=\mathrm{SL}_2(\mathbb{R}) \ltimes \mathbb{R}^2$, where the group operation in the semidirect product is defined as $$[A, b] \cdot [C,d]=[AC, Ad+b].$$  Let $\Gamma:=\mathrm{ASL}_2(\mathbb{Z})$ be the subgroup of $G$ with all entries coming from integers, and set $G':=\mathrm{SL}_2(\mathbb{R})$ where it is viewed as a subgroup of $G$ via the embedding $A \mapsto [A, 0]$. Set $\Gamma'$ to be the subgroup $\mathrm{SL}_2(\mathbb{Z})$ of $G'$. One can identify the space of affine unimodular lattices in $\R^2$ with the space $X:=G/\Gamma$ via the association $$[A, b]\Gamma \mapsto A\mathbb{Z}^2+b \subset \R^2.$$

In the above association, the space $G'/\Gamma'$ (an embedded submanifold of $G/\Gamma$) corresponds to the space of unimodular lattices in $\R^2$. 

It is a standard fact that $\Gamma$ is a lattice in $G$ and $\Gamma'$ is a lattice in $G'$. We denote the unique $G$-invariant probability measure on $G/\Gamma$ by $\mu$ and the unique $ G'$-invariant probability measure on $G'/\Gamma'$ by $\mu'$. We set  $X:=G/\Gamma$ and $X':=G'/\Gamma'$.

For $T\in \R_{>0}$, let $a_T$ be the element of $G' \subset G$ given by $$a_T:= \left[\matr{T}{0}{0}{\frac{1}{T}}, \mt{0}{0}\right]$$ and for $x \in \R$, define $u(x)$ to be the element of $G' \subset G$ given by $$u(x):=\left[\matr1x01, \mt00 \right].$$

Let $A\subset G$ be the subgroup $A:=\{a_T: T>0\}$ of $G' \subset G$, and let $A^{+} \subset A$ be the semigroup consisting of elements $\{a_T: T\geq1\}$. Also, let $U$ be the group $U:=\{u(x):x \in \R\}$. At times we also denote the coset $\bigg[1_2, \mt{\xi}{0}\bigg]\Gamma \in G/\Gamma$ by $[1_2, \bm{\xi}]\Gamma$ for the sake of notational brevity (i.e. set $\bm{\xi}=\mt{\xi}{0}$). 

\subsection{Method of proof}
Our approach uses homogeneous dynamics. Namely, we reduce the Diophantine problem to the study of expanding translates of certain unipotent orbits on the moduli space of unimodular affine lattices in two dimensions.  In our setup, we deal with a significantly more delicate situation because we do not have the full horospherical group anymore but a one-dimensional slice of it, and the equidistribution of the slice is more involved. Already the plain equidistribution statement requires Ratner's measure classification result. See \cref{stromberg} for related discussion. The Diophantine condition on $\xi$ supplies the quantitative control required for effective equidistribution. From this single effective equidistribution statement we derive a multiple effective equidistribution statement, which we use to prove a Poisson approximation theorem for shrinking subsets of $G/\Gamma$. This in turn proves the extreme value laws.

\subsection{Structure of the paper} The remainder of this paper is structured as follows. We begin in \cref{sec:vol} by computing the essential measure estimates for our target sets in $G/\Gamma$ and $G'/\Gamma'$, which we then regularize in \cref{sec:smooth_appro} by constructing smooth approximations to these target sets. Section \ref{sec:effe_equi} is devoted to establishing our core dynamical tool, i.e., the effective multiple equidistribution result. In \cref{sec:poi} we establish a Poisson limit law for shrinking targets on $G/\Gamma$ using the method of moments. In \cref{sec:wei} we prove \thmref{thm:wei} using the Poisson limit law. In \cref{sec:homo} we sketch a method to derive \thmref{homogeneous}. Finally in \cref{sec:loglaw} we establish \corref{cor1} and \corref{cor2}. 

\section{Volume estimates}\label{sec:vol}
In this section we estimate the growth of the volumes of certain sets in $G/\Gamma$ and $G'/\Gamma'$. We use the volume computations done in \cite[Section 8]{stromsmall} to deduce the volume asymptotics in most cases. These asymptotics will be used crucially later.
\subsection{Volumes in $G/\Gamma$}\label{inhomovol}

For $t\geq0$, consider the region $E_t \subset \R^2$ given by \begin{equation}\label{eucregion}
E_t:=\left\{ {\mt{x}{y}} \in \R^2: |x|\leq t, \;  0 \leq y\leq 1 \right\}.
\end{equation} 

Given $E_t \subset \R^2$ as above and $r \in \mathbb{N}$ we define the regions $A_t^{(r)} \subset G/\Gamma$ by 
\begin{equation}\label{regions}
A_t^{(r)}:=\{ \mathcal{L} \in G/\Gamma: \#(\mathcal{L} \cap E_t) \geq r \}.
\end{equation}   

\begin{lem}\label{r=1 vol}The volume of the set $A_t^{(1)}$ satisfies:
    $$2t(1-2t)\leq\mu(A_t^{(1)}) \leq 2t$$ for all $0 \leq t \leq \frac{1}{2}$.\
\end{lem}
\begin{proof} For the upper bound note that $$\mu(A_t^{(1)})\leq \int_{G/\Gamma} \hat{\chi}_{E_t}(\mathcal{L})d\mu\overset{\text{Siegel}}{=}\int_{\R^2}\chi_{E_t} d\lambda=\textnormal{Vol}(E_t)=2t.$$ To show the lower bound recall that Rogers' second moment formula for $X=G/\Gamma$ (see \cite[Appendix B]{BMV}) gives us $$\int_{G/\Gamma}\hat{\chi}^2_{E_t} (\mathcal{L})d\mu(\mathcal{L})=\textnormal{Vol}(E_t)^2+\textnormal{Vol}(E_t).$$

Applying the Cauchy-Schwarz inequality we have $$\Bigg(\int_{G/\Gamma}\hat{\chi}_{E_t} d\mu\Bigg)^2 =\Bigg(\int_{G/\Gamma} \hat{\chi}_{E_t}\cdot \chi_{A_t^{(1)}} d\mu \Bigg)^2\leq \int_{G/\Gamma} \hat{\chi}^2_{E_t}d \mu \int_{G/\Gamma} \chi^{2}_{A_t^{(1)}} d\mu.$$
Consequently, $\mu(A_t^{(1)}) \geq \dfrac{\textnormal{Vol}(E_t)^2}{\textnormal{Vol}(E_t)^2+\textnormal{Vol}(E_t)}=\dfrac{2t}{1+2t}\geq2t(1-2t)$ (since for any $x\geq 0$ we have $\frac{1}{1+x}\geq1-x$).
\end{proof}
The above method of directly using the Siegel and Rogers formulas alone seems insufficient to obtain the correct bounds for $\mu(A_t^{(r)})$ when $t$ is small and $r\geq 2$. However, thankfully, volumes of the sets $A_t^{(r)} \subset G/\Gamma$ when $t$ is small have been computed explicitly by Str\"ombergsson and Venkatesh in \cite[Section 8, Proposition 2]{stromsmall}. Their explicit computation will give us the small volume asymptotics we will need. We record their result here in the form needed by us:
\begin{lem}\label{r>=2 vol}
For every $r \geq 2$ and $0\leq t \leq\frac{1}{2}$

$$\mu(A^{(r)}_t)=\frac{12}{\pi^2}\bigg(\frac{1}{(r-1)^2}-\frac{1}{r^2}\bigg)t^2.$$

\end{lem}
\begin{proof}
    Set $a=2t$ in \cite[Proposition 2, Section 8]{stromsmall} and note that $\mu(A_t^{(r)})=\sum_{j\geq r}f_j(2t)$. It is an easy check to see that the sum telescopes to the above expression for $r\geq2$.
\end{proof}

In the case $r=1$, note that \cite[Proposition 2, Section 8]{stromsmall} gives us that \begin{equation}\mu(A_t^{(1)})=2t-\frac{12}{\pi^2}t^2\end{equation}for all $0 \leq t\leq \frac{1}{2}$, which is stronger than \lemref{r=1 vol}. However, the bounds in \lemref{r=1 vol} suffice for our purpose.\\

Let $C^{(0)}_t$ denote the set of all affine lattices that do not intersect the rectangle $E_t$ i.e. \begin{equation}C^{(0)}_t:=\{\Ll \in G/\Gamma: \Ll \cap E_t=\varnothing \}.\end{equation}  Note that $C^{(0)}_t=X \setminus A^{(1)}_t$.  It is a straightforward application of \cite[Theorem 1]{JA} that $\mu(C^{(0)}_t) \to 0$ as $t \to \infty$. We will later need precise asymptotics of $\mu(C^{(0)}_t)$ for large $t$.  We work this out in the following lemma:  
\begin{lem} \label{asymp}
The volume of the set $C^{(0)}_t$ satisfies 
\begin{equation}
    \lim_{t \to \infty} \mu(C^{(0)}_t)t=\frac{1}{\pi^2}.
\end{equation} 
\end{lem}
\begin{proof}
From the discussion following \cite[Section 8, Proposition 2]{stromsmall}, we have a formula for the measure of affine unimodular lattices that do not intersect a rectangle of area $a$. The measure of the set of affine lattices that do not intersect a rectangle of area $a$ is denoted by $f^{\text{box},\mathrm{ASL}_2}_0(a)$ (following the notation in \cite{stromsmall}). Note that $\mu(C^{(0)}_t)=f^{\text{box},\mathrm{ASL}_2}_0(2t)$. For $a\geq 1$ the authors give a formula for $f^{\text{box},\mathrm{ASL}_2}_0(a)$ in terms of the function $\Xi(a)$. We recall that for $a>0$,  $\Xi(a)$ is the function such that $$\Xi^{''}(a)=\left(1-\frac{1}{a}\right)^2\log\left(\left|1-\frac{1}{a}\right|\right); \quad \Xi'(1)=\Xi(1)=0.$$    

Then, by \cite[Section 8]{stromsmall}, one has 
\begin{equation} \label{step0}
f^{\text{box},\mathrm{ASL}_2}_0(a)=\begin{cases}
     \frac{3}{\pi^2}a^2-a+1 & \textnormal{if} \quad 0\leq a\leq 1\\
    \frac{12}{\pi^2}(\Xi(a)-\Xi(a/2))+\frac{6}{\pi^2}a\log a+c_1a+c_2& \textnormal{if} \quad 1\leq a
\end{cases}    
\end{equation}
where $c_1=\frac{6+6\log(2)}{\pi^2}-2$ and $c_2=\frac{18\log(2)}{\pi^2}$.\\

By Taylor expanding the above expression for $\Xi''$ we get that for all $a\geq1$ the series $$\Xi''(a)=-\frac{1}{a}+\frac{3}{2a^2}+\sum_{n\geq 3} \frac{b_n}{a^n}.$$converges uniformly on $[1, \infty)$, where $b_n=\frac{2}{n-1}-\frac{1}{n}-\frac{1}{n-2}$ (since $|b_n|=O(1/n^3)$, the sequence is absolutely summable, thus by $M$-test we get the uniform convergence). Integrating once and using that $\Xi'(1)=0$ one gets a series for $\Xi'(a)$ when $a\geq1$ that converges uniformly on $[1, \infty)$  $$\Xi'(a)=-\log(a)-\frac{3}{2a}-\sum_{n\geq3}\frac{b_n}{(n-1)a^{n-1}}+\sum_{n\geq 3}\frac{b_n}{(n-1)}+\frac{3}{2}.$$
Set $C=\sum_{n\geq 3}\frac{b_n}{(n-1)}+\frac{3}{2}$. A telescoping sum argument gives $C=\pi^2/3-2$. Integrating again, we get \begin{equation}
\Xi(a)=-a\log(a)+a-\frac{3}{2}\log(a)+C(a-1)+\sum_{n\geq3}\frac{b_n}{(n-1)(n-2)a^{n-2}}+D
\end{equation}
where $D=-\sum_{n\geq 3}\frac{b_n}{(n-1)(n-2)}-1.$

Using the above expansion for $\Xi(a)$ one can calculate that for all $a\geq 2$ one has \begin{multline}\label{step1}
    \frac{12}{\pi^2}(\Xi(a)-\Xi(a/2))=-\frac{6}{\pi^2}a\log(a)+\frac{6}{\pi^2}a(1-\log(2)+C) \\
     -\frac{18\log2}{\pi^2}+\frac{12}{\pi^2}(\Theta(a)-\Theta(a/2))
\end{multline} 

where, for $a\geq 1$, we set $\Theta(a)$ to be the function given by \begin{equation}
    \Theta(a):=\sum_{n \geq 3} \frac{b_n}{(n-1)(n-2)a^{n-2}}. 
\end{equation}

Thus by \eqref{step0} and \eqref{step1}, and noting that $c_1=\frac{6\log(2)}{\pi^2}-\frac{6(1+C)}{\pi^2}$ and $c_2=\frac{18\log2}{\pi^2}$ one has that for all $a\geq2$ $$f^{\text{box},\mathrm{ASL}_2}_0(a)=\frac{12}{\pi^2}\left(\Theta(a)-\Theta(a/2)\right).$$ From here it is easy to see that $$\lim_{a \to \infty}f^{\text{box},\mathrm{ASL}_2}_0(a)a=\frac{12}{\pi^2}(b_3/2-b_3)=\frac{2}{\pi^2}.$$ This proves our claim.
\end{proof}
\begin{remark}
     From \lemref{asymp}, one can conclude that for any bounded subset $K \subset \R^2$, there exists a positive measure set of affine unimodular lattices that avoid $K$.
\end{remark}
\subsection{Volumes in $G'/\Gamma'$} Given $t\geq 0$, consider the region $\tilde{E}_t \subset \R^2$ given by \begin{equation}\label{eucregionhom}
\tilde{E}_t:=\left\{ {\mt{x}{y}} \in \R^2: |x|\leq t, \;   |y|\leq 1 \right\}.
\end{equation} 
Clearly $\textnormal{Vol}(\tilde{E}_t)=4t.$
Given $r \geq 1$, set $$B^{(r)}_t:=\{\mathcal{L}\in G'/\Gamma': \#(\Ll \cap \tilde{E}_t)\geq 2r+1\}.$$

Recall the following volume estimate from \cite[Section 8, Proposition 3]{stromsmall}:
\begin{lem} \label{homvol} 
For all $r\geq 1$ and $0\leq t\leq \frac{1}{2}$, one has $$\mu'(B^{(r)}_t)=\frac{12}{\pi^2r^2}t.$$   
\end{lem}
\begin{proof}
    Set $a=4t$ in \cite[Section 8, Proposition 3]{stromsmall}, and note that $\mu^{'}(B_t^{(r)})=\sum_{j\geq r}f_{2j+1}(4t)$. Then it is easy to see that the sum telescopes to the above expression.
\end{proof}

Note that by Minkowski's first theorem, any unimodular lattice has a \emph{non-zero} lattice point in the region $\tilde{E}_1$. So the measure of lattices that have no non-zero lattice points in the region $\tilde{E}_{1-s}$ goes to zero as $s \to 0^{+}$. We would like to quantify the rate at which this quantity goes to zero. We calculate the asymptotics in the lemma below. Measure estimates of this kind are also worked out  in \cite{kleinbock2022measure}.
\begin{lem}\label{homouppervol}
    \begin{equation}
        \lim_{s \to 0^{+}} \frac{\mu'( \{\Ll \in G'/\Gamma': \#(\Ll \cap \tilde{E}_{1-s})=1\})}{s^2 \log(\frac{1}{s})}=\frac{6}{\pi^2}.
    \end{equation}
\end{lem}

\begin{proof}
A formula for the measure of the set of lattices that have only one point in a rectangle centered at the origin of area $a$ is given in the discussion following \cite[Section 8, Proposition 3]{stromsmall}. $f_1^{\text{box}, \mathrm{SL}_2}(a)$ (borrowing notation from \cite{stromsmall}) denotes the measure of the set of lattices in $G'/\Gamma'$ that intersect a rectangle of area $a$ centered at the origin in exactly one point (i.e., only intersect at the origin). Then
\begin{equation}
    f_1^{\text{box}, \mathrm{SL}_2}(a)= \begin{cases}
        1-\frac{3}{\pi^2}a  & \textnormal{if} \quad 0\leq a\leq 2\\
        \frac{12}{\pi^2}(\Psi(a/4)-\log(a)) +\frac{3}{\pi^2}a+\frac{12}{\pi^2}(2\log2-1) & \textnormal{if} \quad 2\leq a \leq4\\
        0 & \textnormal{if} \quad 4\leq a,
    \end{cases}
\end{equation}
   where $\Psi$ is the function on $\R_{>0}$ defined by $$\Psi'(a)=(a^{-1}-1)\log|a^{-1}-1| \quad \textnormal{for all}\; a>0; \quad \Psi(1)=0.$$
   Note that for small $s>0$, one has the expansion 
   \begin{align*}
   \Psi'(1-s)=&\frac{s}{1-s}\log\left(\frac{s}{1-s}\right)\\
   =&(s+s^2+s^3+ \ldots) (\log(s)+s+\frac{s^2}{2}+\frac{s^3}{3}+\ldots)\\
   =&s\log(s)+s^2(\log(s)+1)+O(s^3\log(s)).
   \end{align*}
   Integrating gives, for small $s>0$
   \begin{align*}
       \Psi(1-s)&=-\int_{0}^s\Psi'(1-t)dt\\
       &=-\frac{1}{2}s^2\log(s)+\frac{s^2}{4}+O(s^3\log(s)).
   \end{align*}
   Noting that $\mu'(\{\Ll \in G'/\Gamma': \#(\Ll \cap \tilde{E}_{1-s})=1\})=f_1^{\text{box}, \mathrm{SL}_2}(4-4s)$, one has for small $s>0$
   \begin{align*}
   \mu'(\{\Ll \in G'/\Gamma': \#(\Ll \cap \tilde{E}_{1-s})=1\})&=\frac{12}{\pi^2}(\Psi(1-s) -\log(4)-\log(1-s)) \\
   & \qquad \quad +\frac{12}{\pi^2}(1-s)+\frac{12}{\pi^2}(2\log 2-1) \\
   &=\frac{12}{\pi^2}\Psi(1-s)-\frac{12}{\pi^2}\log(1-s)-\frac{12}{\pi^2}s\\
   &=-\frac{6}{\pi^2}s^2\log(s)+\frac{3}{\pi^2}s^2+\frac{12}{\pi^2}(s+\frac{s^2}{2})-\frac{12}{\pi^2}s +O(s^3\log(s)) \\ &=-\frac{6}{\pi^2}s^2\log(s)+\frac{9}{\pi^2}s^2 +O(s^3\log(s))
   \end{align*}
Thus, $$\lim_{s \to 0^{+}} \frac{\mu'(\{\Ll\in G'/\Gamma': \#(\Ll \cap \tilde{E}_{1-s})=1\})}{s^2\log(\frac{1}{s})}=\frac{6}{\pi^2}.$$\end{proof}

\section{Smooth approximations}\label{sec:smooth_appro}
This section aims to prove lemmas about approximation of sets in $X$ and $X'$ by smooth functions. We adapt the smoothing arguments in \cite[Section 2]{stromsmall} to smoothly approximate target sets in $X$ and $X'$. \subsection{Sobolev norms}
Let $L$ be a real Lie group and $\Lambda \subset L$ be a lattice.  Denote the Lie algebra of $L$ by $\text{Lie}(L)$, and fix a basis $\{X_i\}$ of $\text{Lie}(L)$. Choose the inner product on $\text{Lie}(L)$ such that the vectors $\{X_i\}$ form an orthonormal basis.  This inner product gives rise to a right-invariant Riemannian metric on $L$ obtained by right translation. Denote the distance function obtained via this by $d_L(\cdot , \cdot )$. This metric descends to a metric on the space $L/\Lambda$, giving rise to a distance function $d_{L/\Lambda}(\cdot,\cdot)$ on $L/\Lambda$. Let $U_\varepsilon$ denote the $\varepsilon$-neighbourhood of the identity element in $L$ in the above metric. Observe that for $x, y \in L/\Lambda$ we have $d_{L/\Lambda}(x,y)<\varepsilon$ if and only if $y \in U_{\varepsilon}\cdot x$.

For any $Y \in \text{Lie}(L)$ and $f:L/\Lambda \to \C$ a smooth function, let $\D_Y(f)(x)$ denote the derivative of $f$ at $x$ along the direction $Y$, i.e., $$\D_Y(f)(x):=\lim_{t \to 0} \frac{f(\exp(tY)x)-f(x)}{t}.$$
For every monomial $Z=X_{i_1}^{l_1}X_{i_2}^{l_2} \cdots X_{i_r}^{l_r}$ made out of $\{X_i \}$ we have a differential operator $\D_Z$ defined as $\D_Z:=\D_{X_{i_1}}^{l_1} \circ \D_{X_{i_2}}^{l_2} \circ \cdots \circ {\D_{X_{i_r}}^{l_r}}$ of order $l=l_1+l_2+\cdots +l_r$. Given $k \in \N$, the $\mathcal{C}^{k}$ norm of $f$ is defined by
\begin{equation} 
\|f\|_{\mathcal{C}^k}:=\max_{0 \leq \textrm{ord}(\D)\leq k}\|\D(f)\|_{\infty}
\end{equation} 
where $\D$ runs over all $\D_Z$ where $Z$ is a monomial in  $\{X_i\}$ of order $\leq k$. Note that $\|f\|_{\mathcal{C}^k} \geq \|f\|_{\infty}$. We denote by $C^k_b(L/\Lambda)$ the set of all functions $f \in C^{\infty}(L/\Lambda)$ with all derivatives of $f$ up to order $k$ bounded pointwise (i.e. $\|f\|_{\mathcal{C}^k}<\infty$) and we set $C^{\infty}_b(L/\Lambda):=\cap_{k \in \N} C^{k}_b(L/\Lambda)$. 

\subsection{Smooth subsets} 

\begin{definition}
Let $A$ be a subset of a metric space $Y$ equipped with a measure $\nu$. For $K>0$, we say that $A$ is $K$-smooth if for all $0<\varepsilon\leq1$, one has $\nu(\partial_{\varepsilon} A) \leq K\varepsilon$. Here we have set $\partial_{\varepsilon} A$ to be the ``$\varepsilon$-neighbourhood of the boundary of $A$", defined as $\partial_\varepsilon A:=B_{\varepsilon}(A) \cap B_\varepsilon(Y-A)$. $B_\varepsilon(A)=\cup_{x \in A} B_\varepsilon(x)$, where $B_\varepsilon(x)$ denotes the open ball of radius $\varepsilon$ centered at the point $x$.
\end{definition}

\begin{remark}\label{triv}
    Note that if $A \subset Y$ is $K$-smooth, then so is the set $Y - A$.
\end{remark}

The usefulness of the above definition is clear from the following lemma, which shows that all $K$-smooth sets of $L/\Lambda$ can be approximated with smooth functions with any given degree of accuracy with an appropriate cost on the ${\mathcal{C}^j}$ norm of the approximating function.

\begin{lem} \label{lemma:2}
Let $A$ be a $K$-smooth subset of $L/\Lambda$, and let $\chi_A$ be the characteristic function of $A$. For each $0<\delta \leq \frac{1}{2}$ there exist functions $f_{\delta}^{-}$ and $f_{\delta}^{+}$ in $C^{\infty}(L/\Lambda)$ such that for all $j\geq 1$ one has 

\begin{enumerate}
    \item $0\leq f_{\delta}^{-} \leq \chi_{A} \leq f_{\delta}^{+}\leq 1$
    \item $\| f_{\delta}^{\pm}-\chi_A\|_{L^1}\leq 2K\delta$
    \item $\|f_\delta^{\pm}\|_{\mathcal{C}^j} \leq C_j   \delta^{-j}$, where $C_j$ is a constant independent of $A$ (but dependent on $j$).
\end{enumerate}
\end{lem}
\begin{proof}
    Choose $f^{\pm}_\delta$ as defined in \cite[Lemma 1]{stromsmall}. (1) and (2) follow exactly as in the proof of \cite[Lemma 1]{stromsmall}. Note that the norm used in \cite[Lemma 1]{stromsmall} is the $L^2_k$ Sobolev norm, but we want bounds for the ${\mathcal{C}^j}$ norms. However, note that for a given differential operator $\mathcal{D}$ of order $j$, the estimate (see \cite[Section 2]{stromsmall} to see how $k_\delta$ are chosen) $$\| \mathcal{D}f_{\delta}^{\pm}\|_\infty\leq \int_{L} |\mathcal{D} k_{\delta}(g)| d\mu_{\text{Haar}}(g) \ll_j \delta^{-j} $$ immediately gives the required bound on the $\mathcal{C}^j$ norm of $f^{\pm}_\delta$. 
\end{proof}
\subsection{$A^{(r)}_t$ and $B^{(r)}_t$ are smooth}

The lemma below proves an important property of the sets $A^{(r)}_t \subset G/\Gamma$. 

\begin{lem}\label{smooth}

For $0<t\leq 1$, there exists a uniform $K>0$ such that the sets $A^{(r)}_t$ are $K$-smooth subsets of $G/\Gamma$ for all $r \in \mathbb{N}$ and all $0< t\leq 1$. Moreover, when $t$ is sufficiently large, there exists a $K'$ such that the sets $A^{(r)}_t$ are $K't^2$-smooth.
 
\end{lem}

\begin{proof} The argument here is similar to the argument in \cite[Lemma 10]{stromsmall}. Let $R$ be a generic rectangle in $\R^2$ given by $R:=\biggl\{\mt{x}y: a \leq x \leq b \; \text{and} \; c\leq y\leq d\biggr\}$. Set $A_R^{(r)}:=\{ \mathcal{L} \in G/\Gamma: \#(\mathcal{L} \cap R)\geq r \}$. This extra notation is introduced to simplify the notation in the proof; later, we will set $R=E_t$ for $t>0$.

Let $\mathcal{L} \in \partial_{\varepsilon}(A_R^{(r)})$, then there exist $u_1, u_2 \in U_{\varepsilon} \subset G$ such that $u_1 \Ll \in A_R^{(r)}$ and $u_2\Ll \notin A_R^{(r)}$. Thus, $\Ll$ has  at least $r$ lattice points in $u_1^{-1}R$ and has fewer than $r$ lattice points in $u_2^{-1}R$. Thus $\Ll$ has  at least one lattice point in the set $u_1^{-1}R\setminus u_2^{-1}R$.     \begin{claim}
        There exists $C>0$ (depending only on the metric on $G$ and $M=\sup_{x \in R} \|x\|$) such that for any $u_1, u_2 \in U_{\varepsilon}$, $$u_1 R \setminus u_2R \subset \partial_{C \varepsilon}R.$$ 
    \end{claim}

    \begin{proof}[Proof of claim]
    
    Let $x \in u_1R\setminus u_2R$. Then, $u_1^{-1} x \in R$ and $u_2^{-1} x \notin R$. Consider the line segment joining $u_1^{-1}x$ and $u_2^{-1} x$. It intersects the boundary of the rectangle $R$, call it $y$.

    Then one can estimate the distance between $x$ and $y$ as
    \begin{align*}
        \|x-y\| & \leq \|x-u_1^{-1}x \| +\|u_1^{-1} x-y\| \\
                & \leq \|x-u_1^{-1}x\| + \|u_1^{-1}x -u_2^{-1}x\| \\
                & \leq \|x-u_1^{-1}x\| + \|u_1^{-1}x -x\| + \|x - u_2^{-1}x \| \\
                & \leq (C_1\varepsilon\|x\| + C_2\varepsilon)+ (C_1\varepsilon\|x\| + C_2\varepsilon) + (C_1\varepsilon\|x\| + C_2\varepsilon) \\
                &= 3C_1\varepsilon \|x\| + 3C_2 \varepsilon. 
    \end{align*}
    Note that since $x \in u_{1}R$ with $u_1 \in U_{\varepsilon}$ for $\varepsilon\leq 1$, one has $\|x\| \leq K(M+1)$ for some constant $K$. Set $C=3C_1K(M+1)+3C_2$. Thus, $x \in \partial_{C\varepsilon}R$.     
    \end{proof}
    
    Following the proof of the claim above, note that in the case of the rectangles being $R=E_t$ for $0\leq t\leq1$, the constant $C$ in the above claim can be taken uniformly to be the constant that works for $E_1$. In the case when $t$ is large enough, one can bound the choice of $C$ in the above lemma by $\ll t$.

    Now let us complete the argument. Since $\Ll$ has a lattice point in $u_1^{-1}R \setminus u_2^{-1}R \subset \partial_{C\varepsilon} R$,  the measure of the set of all such $\Ll$ is bounded above by the volume of the set $\partial_{C \varepsilon}R$ by Siegel's mean value formula for affine lattices, i.e., we have
   $$\mu(\partial_{\varepsilon}A_R^{(r)}) \leq \int_{G/\Gamma} \hat{\chi}_{\partial_{C\varepsilon}R} d\mu=\textrm{Leb}(\partial_{C\varepsilon}R)\leq 2C \textrm{perimeter}(R)\varepsilon.$$

Since the perimeter of all $E_t$ for $0 < t\leq1$ is bounded, we have for all $r \in \N$ and $0 < t\leq 1$ $$\mu(\partial_\varepsilon A_t^{(r)}) \leq K\varepsilon$$ where $K$ is independent of $r$ and $0 < t\leq 1$. When $t$ is large enough, note that $\textnormal{perimeter}(E_t)\ll t$ and $C \ll t$. Using the same argument as above, one now gets that $A^{(r)}_t$ is $K't^2$-smooth.
\end{proof}

\begin{lem}\label{smooth homcase}
    There exists a $K>0$ such that the sets $B^{(r)}_t$ are $Kt^{-1}$-smooth for all $r \in \N$ and $0< t \leq 1$.
\end{lem}
\begin{proof}
The argument here is similar to the argument in \cite[Lemma 4]{stromsmall} and \lemref{smooth}. Let $\mathcal{L} \in \partial_{\varepsilon}(B_{t}^{(r)})$, then there exists $u_1, u_2 \in U_{\varepsilon} \subset G'$ such that $u_1 \Ll \in B_{ t}^{(r)}$ and $u_2\Ll \notin B_{ t}^{(r)}$. Thus, $\Ll$ has  at least $2r+1$ lattice points in $u_1^{-1}\tilde{E}_{t}$ and has fewer than $2r+1$ lattice points in $u_2^{-1}\tilde{E}_{t}$. Thus $\Ll$ has  at least one lattice point in the set $u_1^{-1}\tilde{E}_{t}\setminus u_2^{-1}\tilde{E}_{t}$. Note that by an argument similar to the claim in the previous lemma one sees that $u_1^{-1}\tilde{E}_{t}\setminus u_2^{-1}\tilde{E}_{t} \subset \partial_{C\varepsilon}(\tilde{E}_{t})$,  where $C$ only depends on the metric on $G'$ and $M:=\sup_{x \in \tilde{E}_{t}} \|x\|$ (since for any $u \in U_{\varepsilon} \subset G'$ one has $\|ux-x\| \ll\varepsilon\|x\|$ for all $x \in \R^2$). Thus $C$ can be taken uniform for all $\tilde{E}_{t}$ for $0<t \leq 1$ (can choose $C\geq 1$ that works for $\tilde{E}_1$). For small enough $\varepsilon$, the set $\partial_{C\varepsilon}\tilde{E}_{t}$ misses the origin. One can take $\varepsilon$ such that $0<\varepsilon\leq \frac{t}{1000C}$. For such small $\varepsilon$, one has by Siegel's mean value theorem  $$\mu'(\partial_{\varepsilon}B^{(r)}_{t}) \leq \int_{G'/\Gamma'}\hat{\chi}_{\partial_\varepsilon(\tilde{E}_t)}d\mu'=\textnormal{Vol}(\partial_{C\varepsilon}(\tilde{E}_t)) \leq 16C\varepsilon. $$
So for all $\varepsilon\leq 1$, one can bound $\mu'(\partial_{\varepsilon}B^{(r)}_{t}) \leq 16C(\frac{1000C}{t})\varepsilon=Kt^{-1}\varepsilon$. 
\end{proof}
    
\subsection{A Wiener norm estimate}
Here is an approximation lemma for the intervals $[\alpha, \beta] \subset [0,1]$ by smooth functions with bounded Wiener norms. We will need this in \cref{loglawwithoutevl} to prove the logarithm law, so we prove it here. Given a function $h$ on $\T=\R/\Z$, we set the Wiener norm of $h$ to be $$\|h\|_W:=\sum_{n \in \Z} |\hat{h}(n)|, $$
where $\hat{h}(n):=\int_{0}^{1}h(x)e^{-2\pi inx}dx$.
\begin{lem} \label{apr}
    Let $[\alpha, \beta]$ be a subinterval of $\T=\R/\Z$. Given a sufficiently small $\delta>0$ we have a smooth function $\varphi_\delta$ with the following properties:

    \begin{enumerate}
        \item $\varphi_\delta: \R/\Z \to [0,1]$ is smooth.
        \item $\varphi_\delta=1$ on $[\alpha, \beta]$ and $0$ outside a $2\delta$ neighbourhood of $[\alpha, \beta]$.
        \item There exists a constant $K>0$ (independent of $[\alpha, \beta]$) such that $\|\varphi_\delta\|_W \leq K \delta^{-1}$.
    \end{enumerate}
\end{lem}
\begin{proof}
    Let $g:\R \to \R$ be a $C_c^{\infty}(\R)$ function such that $g\geq0$, $\mathrm{supp}(g)\subset[-1,1]$ and $\int_{\R}g=1$. Set $g_\delta:=\frac{1}{\delta}g(\frac{x}{\delta})$. Clearly, $g_\delta$ is supported in $[-\delta,\delta]$ and $\int_\R g_\delta=1$. For $\delta<1/2$, periodise $g_\delta$ to be a smooth 1-periodic on function on $\R$. Define $\varphi_\delta(x):=(\chi_{[\alpha-\delta, \beta+\delta]}\star g_\delta)(x)$. Then it is easy to see that $\varphi_\delta$ is smooth and satisfies (2).\\

    To see that (3) is also satisfied, note for non-zero $n\in \Z$, the Fourier coefficients of a characteristic function satisfy
    
    \begin{equation}\label{a}|\hat{\chi}_{[\alpha-\delta, \beta+\delta]}(n)|\leq \frac{C}{|n|}.\end{equation}

    Moreover, $\hat{g}_\delta(n)=\hat{g}(\delta n)$. Since $g$ is smooth, $\hat{g}(\xi)$ decays faster than any polynomial in $|\xi|$. In particular, for non-zero $n\in \Z$ (just assuming $g$ is $C^1$ is enough for this decay) we have, \begin{equation}\label{b}|\hat{g}_\delta(n)|=|\hat{g}(\delta n)|\leq \frac{C'}{\delta|n|}.\end{equation}

    Combining \eqref{a}, \eqref{b} and noting that $\hat{\varphi}_\delta(n)=\hat{\chi}_{[\alpha-\delta, \beta+\delta]}(n)\hat{g}_\delta(n)$ we have for non-zero $n \in \Z$ $$|\hat{\varphi}_\delta(n)|\leq \frac{C''}{|n|^2\delta}.$$

    So, $\|\varphi_\delta\|_W=\sum_{n\in \Z} |\hat{\varphi}_\delta(n)| \leq K \delta^{-1}$.\end{proof}

\section{Effective multiple equidistribution}\label{sec:effe_equi}
In this section, we apply \cite[Theorem 2.1]{BMGA} to get an effective multiple equidistribution result suitable for our purpose. 
\subsection{An effective equidistribution result}\label{stromberg}
Given a $\xi \in \T$, note that the $U$ orbit of the coset $\Big[1_2, \mt{\xi}{0}\Big]\Gamma$ is periodic with period $1$. Let $\nu$ be the pushforward of the Lebesgue measure supported on $[0,1]$ to the periodic $U$ orbit, i.e., for any bounded continuous function $f$ on $X$, we have $$\int_{X} f \;d\nu=\int_{0}^{1}f\Big(u(x)\Big[1_2, \mt{\xi}{0}\Big]\Gamma\Big)dx.$$ 
One would like to understand the behaviour of the measures $(a_T)_{*}(\nu)$ on $X$ as $T \to \infty$ (here $(a_T)_{*}(\nu)$ is the pushforward measure of $\nu$ by the action of $a_T$). By an application of Ratner's measure classification theorem, one has that for $\xi \in \T$ irrational \begin{equation}
\label{rat}
(a_T)_{*}(\nu) \xrightarrow{\textnormal{w*}}\mu \quad \textnormal{as} \; T \to \infty.\end{equation}
See \cite[Theorem 5.2]{marklof_strombergsson_distribution_2010} for a proof of \eqref{rat} assuming a result of Shah \cite[Theorem 1.4]{shah1996limit}. 

Given \eqref{rat}, it is natural to ask about the rate of convergence in \eqref{rat}. This was answered by Str\"{o}mbergsson in \cite[Theorem 1.2]{strom}. See also recent work of Kim \cite{kim2024effective}, for analogous effective equidistribution results in higher dimensions. The following theorem is a consequence of \cite[Theorem 1.2]{strom}:

\begin{thm}\label{stt}
    For $\xi \in \T$ of Diophantine type, there exists $C>0$ and $\delta_1>0$ ($C$ and $\delta_1$ depending only on $\xi$) such that for all $0\leq\alpha< \beta \leq1$ and $f \in C^{8}_b(X), \; T\geq 1$ we have 
    \begin{equation}\label{st}
    \Bigg|\frac{1}{\beta-\alpha}\int_{\alpha}^{\beta}f\Big(a_Tu(x)\Big[1_2, \mt{\xi}{0}\Big]\Gamma\Big)\,dx-\int_{X}f\,d\mu \Bigg|\leq \frac{C}{(\beta-\alpha)T^{\delta_1}}\|f\|_{\mathcal{C}^8}.
    \end{equation}
\end{thm}
A few remarks are in order:
\begin{enumerate} 

\item Note that the norm $\| \cdot\|_{\mathcal{C}^8}$ used in this article is comparable to the norm $\| \cdot \|_{C^8_b}$ used in \cite{strom} (i.e., there exist constants $c, C>0$ such that $c\|f \|_{\mathcal{C}^8}\leq \|f \|_{C^{8}_b} \leq C \| f \|_{\mathcal{C}^8}$). So one can replace the norm $\| \cdot\|_{C^8_b}$ by $\|\cdot\|_{\mathcal{C}^8}$ up to a loss of a constant, as we have done in \eqref{st}. 

\item  Moreover if $\xi \in \T$ is of Diophantine type $M\geq 2$, then so is the vector $\bm{\xi}:=\mt{\xi}{0}$. Hence, the discussion after \cite[Theorem 1.2]{strom} applies in our case, which gives \eqref{st}.
\end{enumerate}
\subsection{Spectral gap and effective mixing}
It is well known that the $G'$ action on $L_0^2(G/\Gamma)$ is exponentially mixing (see \cite[Theorem 4.5]{KWa}). Thus, we obtain the following:

\begin{thm} \label{mixing}
There exists a $l \in \N$, and $\tau, E >0$ such that for all $\psi, \varphi \in C^{2, l}(X)$ and $x \in \R$ we have \begin{equation}
    \Bigg|\int_{X} \psi(u(-x)y)\varphi(y)\,d\mu(y)-\int_{X}\psi  \,d\mu\int_{X}\varphi \,d\mu \Bigg|\leq E \, \textnormal{max}\{{1, |x|}\}^{-\tau} \|\psi\|_{L^2_l}\|\varphi\|_{L^2_l}.
\end{equation}
\end{thm}

Here $C^{2,l}(X)$ denotes the subset of smooth functions $f \in L^2(X)$ such that for a fixed $l\in \N$ we have $$\|f\|^2_{L^2_l}:=\sum_{0\leq\textnormal{ord}(Z)\leq l}\|\mathcal{D}_Z(f)\|_2^{2}<\infty,$$ where the above sum runs through monomials in $\{X_1, \ldots, X_5\}$ of order less than or equal to $l$.

\subsection{Effective multiple equidistribution}
In this subsection, we prove effective multiple equidistribution of expanding translates (Theorem \autoref{multiple}) applying \cite[Theorem 2.1]{BMGA}. Before going into details, we set some notation. 

\subsubsection{Notation}\label{notations}
Set $k:=\max\{8,l\}$ (where $l$ is as in Theorem \autoref{mixing}). For a function $\varphi$ on $\T$, recall that the Wiener norm of $\varphi$ is set to be $$\|\varphi\|_{W}:=\sum_{n\in \Z} |\hat{\varphi}(n)|$$ where we have $\hat{\varphi}(n):=\int_{0}^1\varphi(y)e^{-2\pi iny}\,dy$. Let $W(\textnormal{Leb})$ denote the set of all functions on $\T$ such that $\|\varphi\|_W<\infty$.

Given an element $a_T \in A$ define $$\|a_T\|_{*}:=\textnormal{max}\left\{T, \frac{1}{T}\right\}.$$ By abuse of notation, we will denote $\|a_T\|_{*}$ as $\|T\|_{*}$ from here onwards. Clearly $\|T\|_{*} \geq 1$ for all $a_T \in A$. Given $(a_{T_1}, a_{T_2}, \ldots, a_{T_r}) \in A^r$, define
$$m_r(a_{T_1}, a_{T_2}, \ldots, a_{T_r}):=\min_{1\leq i\neq j\leq r}\left\{\left\|\frac{T_i}{T_j}\right\|_{*}\right\} $$ and set $$\rho_{r}(a_{T_1}, a_{T_2}, \ldots, a_{T_r}):=\min_{i=1,\dots,r}\left\{\|T_i\|_{*}\right\}.$$ For $r\geq 2$, we set 
\begin{equation}\label{equ:delta_r}
    \mathfrak{D}_r(a_{T_1}, a_{T_2}, \ldots, a_{T_r}):=\min \left\{\rho_{r}(a_{T_1}, a_{T_2}, \ldots, a_{T_r} ), \;m_r(a_{T_1}, a_{T_2}, \ldots, a_{T_r})\right\}.
\end{equation}
 When $r=1$, we set $\mathfrak{D}_r(a_{T_1})=\|T_1\|_*$. Strictly speaking, note that the definition of $\mathfrak{D}_r$ given above is slightly different from the definition of $\Delta_r$ given in \cite[Section 3.3 \& Section 2]{BMGA}, but they are both the same up to a power. The following is the main theorem of this section:
\begin{thm} \label{multiple}
      Fix $\xi \in \T$ of Diophantine type. Then for every $r\geq1$ there exists $D_r>0$ and $\delta_r>0$ (which only depend on $\xi$ and r) such that for all $f_1, f_2, \ldots,f_r \in C^{\infty}_b(X)$ and $\varphi \in W(\textnormal{Leb})$, we have for every $(a_{T_1}, a_{T_2}, \ldots, a_{T_r})\in (A^{+})^r $ 
\begin{multline}
        \Bigg|\int_{0}^1 \varphi(x) \prod_{j=1}^rf_j(a_{T_j}u(x)[1_2, \bm{\xi}]\Gamma) \,dx -\int_{0}^1 \varphi \,dx \prod_{j=1}^r\int_{X}f_j \,d\mu \Bigg|  \\
        \leq D_r \mathfrak{D}_r(a_{T_1}, a_{T_2}, \ldots, a_{T_r})^{-\delta_r} \|\varphi\|_{W}\prod_{j=1}^r\|f_j\|_{\Cc^k}  .
\end{multline}
\end{thm}
\begin{proof} As stated before, Theorem \ref{multiple} is an application of \cite[Theorem 2.1]{BMGA}. Note that our setup is exactly the same as the one demanded in the first paragraph of \cite[Section 2]{BMGA}. In order to apply \cite[Theorem 2.1]{BMGA} in our case, we have to check the hypotheses (EQ1) and (EQ2) of \cite[Section 2]{BMGA} and choose norms $S_i$ such that they satisfy the conditions (S1)-(S6) of \cite[Section 3.2]{BMGA}. 

For our purpose we set $\mathcal{A}:=C^{\infty}_b(X)$ and the Sobolev norms $S_i$ to be the $\|\cdot\|_{\Cc^k}$ norm for all $i \in \N$ (where $k$ is as defined in \ref{notations}). It is a routine verification to check properties (S1)-(S6) of the Sobolev norms $S_i:=\|\cdot\|_{\Cc^k}$. 

For verifying the hypotheses (EQ1) and (EQ2) of \cite[Section 2]{BMGA} in our case, first note that $\mathcal{A} \subset C^{8}_b(X)$ and $\mathcal{A} \subset C^{2, \infty}$. Moreover, for $f \in \mathcal{A}$ we have $\|f\|_{\mathcal{C}^8} \leq \|f\|_{\Cc^k}$ and $\|f\|_{L^2_l} \leq M \|f\|_{\Cc^k}$ for some constant $M>0$. Now, hypotheses (EQ1) and (EQ2) follow from \corref{cor} and \corref{mixres}, which in turn follow from \thmref{stt} and \thmref{mixing}, respectively:
\begin{cor} \label{cor}
For $\xi \in \T$ of Diophantine type, there exist $C>0$ and $\delta_1>0$ ($C$ and $\delta_1$ depending only on $\xi$) such that for all $f \in \mathcal{A}, \; T\geq 1$ we have 
    \begin{equation*}
    \Bigg|\int_{0}^{1}f\Big(a_Tu(x)\Big[1_2, \mt{\xi}{0}\Big]\Gamma\Big)\,dx-\int_{X}f\,d\mu \Bigg|\leq \frac{C}{T^{\delta_1}}\|f\|_{\mathcal{C}^k}.
    \end{equation*}
\end{cor}

\begin{cor}\label{mixres}
There exists $E, \tau >0$ such that for all $f_1, f_2 \in \mathcal{A}$ and $x \in \textnormal{Lie}(U)$ we have \begin{equation}
    \Bigg|\int_{X} f_1(\exp{(x)}y)f_2(y)\,d\mu(y)-\int_{X}f_1 \, d\mu\int_{X}f_2 d\mu \Bigg|\leq E\,\textnormal{max}\{{1, \|x\|}\}^{-\tau} \|f_1\|_{\Cc^k}\|f_2\|_{\Cc^k}.
\end{equation}
\end{cor}
\end{proof}

\section{Poisson law for subsets of $X$}\label{sec:poi}

\subsection{Poisson approximation}
We first recall a criterion (stated and proved in \cite[Section 4.1]{bjorklund2023poisson}) to determine when a sequence of random variables converges in distribution to a Poisson distribution. We first set up some notation to state the lemma.

Let $(Z, \nu)$ be a probability space and let $H$ be an infinite set. Let $F_n$ be a sequence of finite subsets of $H$ such that $|F_n|\to \infty$. Suppose that for every $n$ and $h \in F_n$ we are given a measurable subset $A_{n, h} \subset Z$. Define \begin{equation}
    N_n(z):=\#\{h\in F_n: \; z \in A_{n,h}\}, \; \text{for} \; z \in Z.
\end{equation}

For $m\geq 0$, and $n, r \in \mathbb{N}$ we define $\Lambda_{n, r}(m)$ to be as follows
\begin{equation}
    \Lambda_{n,r}(m):=\textnormal{max}\bigg\{\big||F_n|^r\nu\big(\bigcap_{h \in F^{'}}A_{n, h}\big)-m^r\big|: F^{'} \subset F_n, \; |F^{'}|=r\bigg\}.
\end{equation}
If $|F_n|<r$, we set $\Lambda_{n, r}(m)=0$.

The following lemma is a straightforward application of the method of moments:
\begin{lem}\label{Poisson} 
If there exists a $m > 0$ such that for all $r\geq 1$ we have 

\begin{equation}
    \lim_{n \to \infty} \Lambda_{n, r}(m)=0.
\end{equation}
Then $$N_n \xLongrightarrow[\nu]{} \textnormal{Poi}(m) \quad \text{as} \; \; n\to \infty.$$
\end{lem}
\begin{proof}
    See \cite[Lemma 4.1]{bjorklund2023poisson} for the proof.
\end{proof}

\subsection{Effective multiple equidistribution implies Poisson law}
In this subsection, we prove a theorem analogous to \cite[Theorem 4.4]{bjorklund2023poisson} in our context. One key difference is that in our case we allow for $A_n$ to be a broader class of subsets than that allowed in \cite{bjorklund2023poisson} (i.e., our subsets $A_n$ need not be such that $A_n$ or $A_n^c$ is compact). We will need this level of generality for our application. 

For our purpose, we take 
$$(Z,\nu):=([0, 1], \textnormal{Leb}) \quad \text{and} \quad H:=A^{+}.$$
Given a finite subset $F :=\{a_{T_1}, \ldots, a_{T_s}\} \subset A^{+}$ of  cardinality $s$, we define:
\begin{equation}\label{w defn}
w(F):=\log(\mathfrak{D}_s(a_{T_1}, \ldots, a_{T_s})),\end{equation}
where $\mathfrak{D}_s$ is as in \eqref{equ:delta_r}.
It is easy to see that for any subset $F^{'}\subset F$ we have $w(F)\leq w(F^{'})$.
\begin{thm}\label{Poi law}
Fix $\xi \in \mathbb{T}$ of Diophantine type. Let $F_n$ be a sequence of finite subsets of $A^{+}$ such that 
\begin{equation}\label{equ:hyp1}
    |F_n| \to \infty \quad \text{and} \quad \frac{w(F_n)}{\log(|F_n|)}\to \infty.
\end{equation}
  Suppose $\{A_n\}_n$ is a sequence of subsets of $G/\Gamma$ that are $K$-smooth for all sufficiently large $n$, where $K$ is uniform. If there exists $m >0$ such that \begin{equation}\label{equ:hyp2}
      \lim_{n \to \infty} \mu(A_n)|F_n|=m,
  \end{equation}  then we have 
   $$N_n \xLongrightarrow[\textnormal{Leb}]{}\textnormal{Poi}(m),$$ i.e., $$\lim_{n\to \infty} \textnormal{Leb}\bigl(\bigl\{\alpha \in \mathbb{T}: \#\{h \in F_n: hu(\alpha)[1_2, \bm{\xi}]\Gamma \in A_n\}=k\bigr\}\bigr)=\frac{m^ke^{-m}}{k!}.$$
\end{thm}
\begin{proof}
Since the sets $A_n$ are $K$-smooth, set the $\delta_n$ in \lemref{lemma:2} for each $A_n$ to be $1/|F_n|^2$. So for large enough $n$ there exist smooth functions $\varphi_n^{+}$ and $\varphi_n^{-}$ such that 
\begin{equation}\label{equ:appro_A_n}
    0\leq\varphi_{n}^{-}\leq\chi_{A_n}\leq \varphi_{n}^{+}\leq 1,\,\, \|\chi_{A_n}-\varphi_n^{\pm}\|_{L^1}\leq \frac{2K}{|F_n|^{2}}, \,\, \textnormal{and} \,\, \|\varphi_n^{\pm}\|_{\Cc^k} \leq C_k|F_n|^{2k}.
\end{equation}

For any $F^{'} \subset F_n$ such that $|F^{'}|=r$ we have \begin{multline*}
\int_{0}^{1}\prod_{h\in F^{'}}\varphi_n^{-}(hu(\alpha)[1_2, \bm{\xi}]\Gamma)\,d\alpha
\leq \textnormal{Leb}\big(\big\{\alpha \in [0, 1]: \; u(\alpha)[1_2, \bm{\xi}]\Gamma \in \cap_{h \in F^{'}}h^{-1}A_n \big\}\big) \\
\leq \int_{0}^{1}\prod_{h\in F^{'}}\varphi_n^{+}(hu(\alpha)[1_2, \bm{\xi}]\Gamma)d\alpha.
\end{multline*}
Applying \thmref{multiple} we have
\begin{multline*}
\biggl(\int_{X}\varphi_n^{-}\biggr)^r - D_r(C_k)^{r}|F_n|^{2kr}e^{-\delta_rw(F^{'})}\\
\leq \textnormal{Leb}\big(\big\{\alpha \in [0, 1]: \; u(\alpha)[1_2, \bm{\xi}]\Gamma \in \cap_{h \in F^{'}}h^{-1}A_n \big\}\big) \\
\leq \biggl(\int_{X}\varphi_n^{+}\biggr)^r + D_r(C_k)^{r}|F_n|^{2kr}e^{-\delta_rw(F^{'})}.\end{multline*}
Multiplying by $|F_n|^r$, subtracting $m^r$, and applying \eqref{equ:appro_A_n}  we get 
\begin{multline*}
\biggl(|F_n|\mu(A_n)-\frac{2K}{|F_n|}\bigg)^r - m^r-D_r(C_k)^{r}|F_n|^{2kr+r}e^{-\delta_rw(F^{'})}\\
\leq |F_n|^r\textnormal{Leb}\big(\big\{\alpha \in [0, 1]: \; u(\alpha)[1_2, \bm{\xi}]\Gamma \in \cap_{h \in F^{'}}h^{-1}A_n \big\}\big) -m^r\\
\leq \biggl(|F_n|\mu(A_n)+\frac{2K}{|F_n|}\bigg)^r-m^r + D_r(C_k)^{r}|F_n|^{2kr+r}e^{-\delta_rw(F^{'})}
\end{multline*}
Noting that $e^{-\delta_rw(F^{'})}\leq e^{-\delta_rw(F_n)}$ for all subsets $F^{'}\subset F_n$ of size $r$, we have 
\begin{multline*}
\max\bigg\{\bigg||F_n|^r\textnormal{Leb}\big(\big\{\alpha \in [0, 1]: \; u(\alpha)[1_2, \bm{\xi}]\Gamma \in \cap_{h \in F^{'}}h^{-1}A_n \big\}\big) -m^r\bigg|:F^{'}\subset F_n, |F^{'}|=r \bigg\}\\
\leq \max\bigg\{\bigg|\biggl(|F_n|\mu(A_n)+\frac{2K}{|F_n|}\bigg)^r-m^r + D_r(C_k)^{r}|F_n|^{2kr+r}e^{-\delta_rw(F_n)}\bigg|, \\ 
\bigg|\biggl(|F_n|\mu(A_n)-\frac{2K}{|F_n|}\bigg)^r - m^r-D_r(C_k)^{r}|F_n|^{2kr+r}e^{-\delta_rw(F_n)}\bigg|\bigg\}.
\end{multline*}
Now for a fixed $r$,  by hypothesis \eqref{equ:hyp1} and \eqref{equ:hyp2} we have that both $$\bigg|\biggl(|F_n|\mu(A_n)+\frac{2K}{|F_n|}\bigg)^r-m^r + D_r(C_k)^{r}|F_n|^{2kr+r}e^{-\delta_rw(F_n)}\bigg| \to 0$$ and
$$\bigg|\biggl(|F_n|\mu(A_n)-\frac{2K}{|F_n|}\bigg)^r - m^r-D_r(C_k)^{r}|F_n|^{2kr+r}e^{-\delta_rw(F_n)}\bigg| \to 0$$ as $n \to \infty$. 

Thus, for each $r$ we have  
\begin{equation*}
\lim_{n\to \infty}\max\bigg\{\bigg||F_n|^r\textnormal{Leb}\big(\big\{\alpha \in [0, 1]: \; u(\alpha)[1_2, \bm{\xi}]\Gamma \in \cap_{h \in F^{'}}h^{-1}A_n \big\}\big) -m^r\bigg|:F^{'}\subset F_n, |F^{'}|=r \bigg\}=0.
\end{equation*}
Thus, by \lemref{Poisson} we have $N_n \xLongrightarrow[\textnormal{Leb}]{}\textnormal{Poi}(m)$.
\end{proof}

Note that we have used the assumption that eventually all $A_n$ are $K$-smooth subsets of $G/\Gamma$ in the above proof. By following the above proof it should be clear that one can weaken the hypothesis on $A_n$ a little by removing that uniformity. Namely, one can prove the following stronger variant (which will be of use to us):

\begin{thm}\label{Poi law refined}
Let $\xi$ and $F_n$ be as in \thmref{Poi law}. Suppose $\{A_n\}_n $ is a sequence of subsets of $G/\Gamma$ that are $K|F_n|^s$-smooth eventually (i.e., for a fixed $K$ and $s$, the sets $A_n$ are $K|F_n|^s$-smooth for all sufficiently large $n$). If there exists $m>0$ such that \begin{equation}\label{equ:hyp2spec}
      \lim_{n \to \infty} \mu(A_n)|F_n|=m,
  \end{equation}  then we have 
   $$N_n \xLongrightarrow[\textnormal{Leb}]{}\textnormal{Poi}(m),$$ i.e., $$\lim_{n\to \infty} \textnormal{Leb}\bigl(\bigl\{\alpha \in \mathbb{T}: \#\{h \in F_n: hu(\alpha)[1_2, \bm{\xi}]\Gamma \in A_n\}=k\bigr\}\bigr)=\frac{m^ke^{-m}}{k!}.$$
\end{thm}
\begin{proof}
    Choose the $\delta_n$ in the proof of \thmref{Poi law} to be $\delta_n:=1/|F_n|^{s+2}$ and run exactly the same argument as in \thmref{Poi law}.
\end{proof}
\section{Proof of \thmref{thm:wei}}\label{sec:wei}

In this section, we prove the Weibull limit laws for minima and the Fr\'echet limit law for maxima in the case when $\xi$ is of Diophantine type. The proof technique is similar to the method adopted in \cite{bjorklund2023poisson}, with the key input being \thmref{Poi law} and  \thmref{Poi law refined} combined with the volume asymptotics of subsection \ref{inhomovol}.
\subsection{An observation} \label{obs}
Observe that for any $t>0$ and $T>2t$, one has $k^{(r)}_\xi(T;\alpha) \leq t$ if and only if $a_Tu(\alpha)\Big[1_2, \mt{-\xi}{0}\Big]\Gamma \in A^{(r)}_t$. So, for example, for the subsets $\Delta_n \subset [1, \infty)$ as in the hypothesis of \thmref{thm:wei} one has for large enough $n$ $$\mathfrak{m}_\xi^{(r)}(\Delta_n; \alpha) \leq t \Leftrightarrow \#\left\{T\in \Delta_n: a_Tu(\alpha)\left[1_2, \mt{-\xi}{0}\right]\Gamma \in A^{(r)}_t\right\}>0.$$
\subsection{Proof of \thmref{thm:wei}}
\begin{proof}[Proof of \thmref{thm:wei}]
\underline{\textbf{Case 1:} $r=1$ minima}\\
Given the sequence $\Delta_n$, set $F_n \subset A^{+}$ to be the sequence of subsets of $A^{+}$, where $F_n=\{a_T:T\in \Delta_n\}$. Given such a $\Delta_n$ and a $t>0$, set $A_n:=A^{(1)}_{t/|\Delta_n|}$ for $n \in \N$. Then using \lemref{r=1 vol} one has $$2t\bigg(1-\frac{2t}{|\Delta_n|}\bigg)\leq|\Delta_n|\mu(A_n)\leq 2t$$ for all sufficiently large $n$  (i.e. when $\frac{t}{|\Delta_n|} \leq \frac{1}{2}$). Letting $n \to \infty$ we have $\lim_{n \to \infty}\mu(A_n)|\Delta_n|=2t$. Using \lemref{smooth} we conclude that the sets $A_n$ are eventually $K$-smooth. The conditions on the sets $\Delta_n$ are such that the sets $F_n$ satisfy the hypothesis in \thmref{Poi law}. Note that $\xi$ is of Diophantine type if and only if $-\xi$ is of Diophantine type. So, by applying \thmref{Poi law}, we get that for a fixed $\xi$ of Diophantine type one has $$\lim_{n \to \infty} \textnormal{Leb}\bigg(\bigg\{\alpha \in \T: \#\{T \in \Delta_n: a_Tu(\alpha)[1_2, -\bm{\xi}]\Gamma \in A_n\}=k \bigg\}\bigg)=\frac{(2t)^ke^{-2t}}{k!}.$$ By summing over all $k\geq 1$ one has 
$$\lim_{n \to \infty} \textnormal{Leb}\bigg(\bigg\{\alpha \in \T: \#\{T \in \Delta_n: a_Tu(\alpha)[1_2, -\bm{\xi}] \Gamma \in A_n\}>0 \bigg\}\bigg)=1-e^{-2t}.$$
Note that for large enough $n$, the set $$\bigg\{\alpha \in \T: \#\{T \in \Delta_n: a_Tu(\alpha)[1_2, -\bm{\xi}] \Gamma \in A_n\}>0 \bigg\}$$ is the same as the set $$\bigg\{\alpha \in \T: \mathfrak{m}_\xi^{(1)}(\Delta_n; \alpha)\leq\frac{t}{|\Delta_n|} \bigg\}.$$
Therefore one has for all $t>0$ $$\lim_{n \to \infty} \textnormal{Leb}\bigg(\bigg\{\alpha \in \T: \mathfrak{m}_\xi^{(1)}(\Delta_n; \alpha)\leq\frac{t}{|\Delta_n|} \bigg\}\bigg)=1-e^{-2t}.$$

Thus $$|\Delta_n|\mathfrak{m}^{(1)}_\xi(\Delta_n; \cdot)\xLongrightarrow[\textnormal{Leb}]{} \textnormal{Wei}(2, 1).$$\\
\underline{\textbf{Case 2:} $r\geq 2$ minima}\\
Given such a $\Delta_n$, set $F_n \subset A^{+}$ to be exactly as in the previous case. For such a $\Delta_n$ and a $t> 0$, set $A_n:=A^{(r)}_{\frac{t}{\sqrt{|\Delta_n|}}}$ for $n \in \N$.  Then from \lemref{r>=2 vol} one has for large enough $n$ $$|\Delta_n|\mu(A_n)=\frac{12}{\pi^2}\bigg(\frac{1}{(r-1)^2}-\frac{1}{r^2}\bigg)t^2.$$ So one has $$\lim_{n \to \infty} |\Delta_n|\mu(A_n)=\frac{12}{\pi^2}\bigg(\frac{1}{(r-1)^2}-\frac{1}{r^2}\bigg)t^2.$$
Denote the constant $\frac{12}{\pi^2}\big(\frac{1}{(r-1)^2}-\frac{1}{r^2}\big)$ by $c_r$. Again from \lemref{smooth} one has that eventually each $A_n$ is $K$-smooth. The conditions on the sets $\Delta_n$ are set so that the sets $F_n$ satisfy the hypothesis in \thmref{Poi law}. So, by applying \thmref{Poi law}, we get that for a fixed $\xi$ of Diophantine type one has $$\lim_{n \to \infty} \textnormal{Leb}\bigg(\bigg\{\alpha \in \T: \#\{T \in \Delta_n: a_Tu(\alpha)[1_2, -\bm{\xi}]\Gamma \in A_n\}=k \bigg\}\bigg)=\frac{(c_rt^2)^ke^{-c_rt^2}}{k!}.$$
By summing over all $k\geq 1$ one has 
$$\lim_{n \to \infty} \textnormal{Leb}\bigg(\bigg\{\alpha \in \T: \#\{T \in \Delta_n: a_Tu(\alpha)[1_2, -\bm{\xi}] \Gamma \in A_n\}>0 \bigg\}\bigg)=1-e^{-c_rt^2}.$$
Note that for large enough $n$, the set $$\bigg\{\alpha \in \T: \#\{T \in \Delta_n: a_Tu(\alpha)[1_2, -\bm{\xi}]\Gamma \in A_n\}>0 \bigg\}$$ is the same as the set $$\bigg\{\alpha \in \T: \mathfrak{m}_\xi^{(r)}(\Delta_n; \alpha)\leq\frac{t}{\sqrt{|\Delta_n|}} \bigg\}.$$

Therefore one has for all $t>0$ $$\lim_{n \to \infty} \textnormal{Leb}\bigg(\bigg\{\alpha \in \T: \mathfrak{m}_\xi^{(r)}(\Delta_n; \alpha)\leq\frac{t}{\sqrt{|\Delta_n|}} \bigg\}\bigg)=1-e^{-c_rt^2}.$$

Thus $$|\Delta_n|^{1/2}\mathfrak{m}^{(r)}_\xi(\Delta_n; \cdot)\xLongrightarrow[\textnormal{Leb}]{} \textnormal{Wei}\left(\frac{12}{\pi^2}\left(\frac{1}{(r-1)^2}-\frac{1}{r^2}\right), 2\right).$$

\noindent\underline{\textbf{Case 3:} $r=1$ maxima}\\
Given such a $\Delta_n$, set $F_n \subset A^{+}$ to be exactly as in Case 1. For such a $\Delta_n$ and a $t> 0$ we set $A_n:=C^{(0)}_{t|\Delta_n|}$. Since $C^{(0)}_t=X-A^{(1)}_t$, one has that $A_n$ are $K't^2|\Delta_n|^2$ smooth for all large enough $n$ by \lemref{smooth} and Remark \ref{triv}. Moreover, by \lemref{asymp} one has $$\lim_{n\to \infty}\mu(A_n)|\Delta_n|=\frac{1}{\pi^2t}.$$

The conditions on $\Delta_n$ and $A_n$ satisfy the hypothesis of \thmref{Poi law refined} (with $s=2$). So one has $$\lim_{n \to \infty} \textnormal{Leb}\bigg(\bigg\{\alpha \in \T: \#\{T \in \Delta_n: a_Tu(\alpha)[1_2, -\bm{\xi}]\Gamma \in A_n\}=k \bigg\}\bigg)=\frac{1}{k!(\pi^2t)^k}e^{-\frac{1}{\pi^2t}}.$$
Summing over $k\geq1$, we have 
$$\lim_{n \to \infty} \textnormal{Leb}\bigg(\bigg\{\alpha \in \T: \#\{T \in \Delta_n: a_Tu(\alpha)[1_2, -\bm{\xi}]\Gamma \in A_n\} >0\bigg\}\bigg)=1-e^{-\frac{1}{\pi^2t}}.$$ By our hypothesis on $\Delta_n$, one has $\min_{T \in \Delta_n}{\{T\}} \geq|\Delta_n|^{2}$ for all large enough $n$. So for all large enough $n$ one has $\frac{t|\Delta_n|}{\min_{T \in \Delta_n}\{T\}} < \frac{1}{2}$. Therefore for large enough $n$, the set $$\bigg\{\alpha \in \T: \#\{T \in \Delta_n: a_Tu(\alpha)[1_2, -\bm{\xi}]\Gamma \in A_n\}>0 \bigg\}$$ is the same as the set $$\bigg\{\alpha \in \T: \mathfrak{M}_\xi^{(1)}(\Delta_n; \alpha)> t|\Delta_n| \bigg\}.$$

So one has for all $t>0$ $$\lim_{n \to \infty} \textnormal{Leb}\bigg(\bigg\{\alpha \in \T: |\Delta_n|^{-1}\mathfrak{M}_\xi^{(1)}(\Delta_n; \alpha)> t \bigg\}\bigg)=1-e^{-\frac{1}{\pi^2t}}.$$

Thus $$|\Delta_n|^{-1}\mathfrak{M}^{(1)}_\xi(\Delta_n; \cdot)\xLongrightarrow[\textnormal{Leb}]{} \textnormal{Fre}\left(\frac{1}{\pi^2}, 1\right).$$
\end{proof}
\section{Proof of \thmref{homogeneous}}\label{sec:homo}
Analogous to the inhomogeneous case, one can prove the homogeneous versions of the above theorems once we know the volume asymptotics and effective multiple equidistribution. In this section, we give a sketch of the arguments needed to deduce the theorems of the homogeneous case. 

On the space $G'/\Gamma'$ we have the analogous effective multiple equidistribution statement.

\begin{thm}[Effective multiple equidistribution] There exists a $k \in \N$ such that for every $r\geq 1$, there exists constants $D'_r,\delta'_r>0$ such that for all $f_1, f_2, \ldots, f_r \in C^{\infty}_b(G'/\Gamma')$ and $\varphi \in W(\textnormal{Leb})$, we have for every $(a_{T_1}, a_{T_2},\ldots, a_{T_r}) \in (A^{+})^r$ 
    \begin{multline}
        \Bigg|\int_{0}^1 \varphi(x) \prod_{j=1}^rf_j(a_{T_j}u(x)\Gamma') \,dx -\int_{0}^1 \varphi \,dx \prod_{j=1}^r\int_{G'/\Gamma'}f_j \,d\mu' \Bigg|  \\
        \leq D'_r \mathfrak{D}_r(a_{T_1}, a_{T_2}, \ldots, a_{T_r})^{-\delta'_r} \|\varphi\|_{W}\prod_{j=1}^r\|f_j\|_{\Cc^k}  .
\end{multline}
\end{thm}
\begin{proof}
We can derive this similarly to \thmref{multiple} by invoking effective (single) equidistribution of expanding horospheres on $G'/\Gamma'$  as stated in \cite[Section 5, eq (23)]{strom}. See \cite[Remark 3.4]{Andreasdeviationofergodic} on how to derive this version of effective single equidistribution. We can bound the error term by a polynomial power in the above version and use \cite[Theorem 2.1]{BMGA} to conclude.    
\end{proof}
Recall the definition of $w(\cdot)$ from \eqref{w defn}. Then by an analogous argument to \thmref{Poi law refined} one can derive the following:
\begin{thm}[Poisson convergence in $G'/\Gamma'$]\label{Poisson homocase}
 Let $F_n$ be a sequence of finite subsets of $A^{+}$ such that 
\begin{equation*}
    |F_n| \to \infty \quad \text{and} \quad \frac{w(F_n)}{\log(|F_n|)}\to \infty.
\end{equation*}
  Suppose $\{A_n\}_n$ is a sequence of subsets of $G'/\Gamma'$ that are $K|F_n|^s$-smooth for all sufficiently large $n$. If there exists $m>0$ such that \begin{equation}\label{equ:homhyp2}
      \lim_{n \to \infty} \mu'(A_n)|F_n|=m,
  \end{equation}  then we have 
   $$N_n \xLongrightarrow[\textnormal{Leb}]{}\textnormal{Poi}(m),$$ i.e., $$\lim_{n\to \infty} \textnormal{Leb}\bigl(\bigl\{\alpha \in \mathbb{T}: \#\{h \in F_n: hu(\alpha)\Gamma' \in A_n\}=k\bigr\}\bigr)=\frac{m^ke^{-m}}{k!}.$$
\end{thm}
\begin{remark}
    A version of \thmref{Poisson homocase} is already established in \cite{bjorklund2023poisson} by combining \cite[Theorem 4.4]{bjorklund2023poisson} and \cite[Theorem 6.4]{bjorklund2023poisson}. 
\end{remark}
\subsection{Weibull limit laws for minima and maxima}

Note that for $t\leq1$, and $T$ sufficiently large one has $k^{(r)}_{0}(T; \alpha) \leq t$ if and only if $a_Tu(\alpha)\Gamma' \in B^{(r)}_t$ (i.e., the lattice $a_Tu(\alpha)\Gamma'$ has  at least $2r+1$ lattice points in the region $\tilde{E}_t$). 

\begin{proof}[Proof of \thmref{homogeneous}]
    \underline{\textbf{Case 1:} $r\geq 1$ minima}\\   
    Given such a sequence $\Delta_n$, we set $F_n$ to be $F_n=\{a_T: T \in \Delta_n\}$. For any $t>0$, we set $A_n:=B^{(r)}_{t/|\Delta_n|}$. Then, for a fixed $t$, one has for large enough $n$ (by \lemref{homvol}) $\mu'(A_n)=\frac{12}{\pi^2r^2}\frac{t}{|\Delta_n|}$. Thus, $$\lim_{n \to \infty} |\Delta_n|\mu'(A_n)=\frac{12}{\pi^2r^2}t.$$
    Moreover, the sets $A_n$ are $Kt^{-1}|\Delta_n|$-smooth by \lemref{smooth homcase}.
    Thus by \thmref{Poisson homocase}, one has \begin{equation}
        \lim_{n \to \infty} \textnormal{Leb}\bigg(\bigg\{\alpha \in \T: \#\{T \in \Delta_n: a_Tu(\alpha)\Gamma' \in A_n\}=k \bigg\}\bigg)=\frac{1}{k!}\left(\frac{12}{\pi^2r^2}t\right)^ke^{-\frac{12}{\pi^2r^2}t}.
    \end{equation}
    Summing over all $k\geq 1$ we have 
    \begin{equation}
        \lim_{n \to \infty} \textnormal{Leb}\bigg(\bigg\{\alpha \in \T: \#\{T \in \Delta_n: a_Tu(\alpha)\Gamma' \in A_n\}>0 \bigg\}\bigg)=1-e^{-\frac{12}{\pi^2r^2}t}.
    \end{equation}
    Note that for large enough $n$, the set $$\bigg\{\alpha \in \T: \#\{T\in \Delta_n:a_Tu(\alpha)\Gamma' \in A_n\}>0\bigg\}$$ is the same as the set $$\bigg\{\alpha \in \T: \mathfrak{m}^{(r)}(\Delta_n; \alpha) \leq \frac{t}{|\Delta_n|}\bigg\}.$$
    Thus, 
    \begin{equation*}
   |\Delta_n|\mathfrak{m}^{(r)}(\Delta_n; \cdot)\xLongrightarrow[\textnormal{Leb}]{} \textnormal{Wei}\left(\frac{12}{\pi^2r^2}, 1\right) \; \textnormal{for all} \;r\geq1 . \end{equation*}
\underline{\textbf{Case 2:} $r=1$ maxima}\\
Note that for any $\Delta_n$, $\mathfrak{M}^{(1)}(\Delta_n; \alpha) \leq 1$ by Dirichlet's theorem. Hence, studying the distribution of $\mathfrak{M}^{(1)}(\Delta_n; \cdot)$ is equivalent to studying $1-\mathfrak{M}^{(1)}(\Delta_n; \cdot)=\min_{T \in \Delta_n} \{1-k^{(1)}_0(T; \cdot)\}$.

For any $t>0$, set $$A_n:=\left\{\Ll \in G'/\Gamma': \Ll \cap \tilde{E}_{1-\frac{t}{{\sqrt{|\Delta_n|\log(|\Delta_n|)}}}}=\{0\}\right\}.$$
By \lemref{homouppervol} one has \begin{align*}\lim_{n \to \infty}\frac{\mu'(A_n)}{\frac{t^2}{|\Delta_n|\log(|\Delta_n|)}\log\left(\frac{|\Delta_n|^{1/2}\log(|\Delta_n|)^{1/2}}{t}\right)}=\frac{6}{\pi^2}.
\end{align*}
Thus, $$\lim_{n \to \infty}|\Delta_n|\mu'(A_n)            =\frac{3t^2}{\pi^2}.$$
Also, for large enough $n$, the sets $A_n$ are uniformly $K$-smooth. Thus by \thmref{Poisson homocase}, one has \begin{equation}
        \lim_{n \to \infty} \textnormal{Leb}\bigg(\bigg\{\alpha \in \T: \#\{T \in \Delta_n: a_Tu(\alpha)\Gamma' \in A_n\}=k \bigg\}\bigg)=\frac{1}{k!}\left(\frac{3t^2}{\pi^2}\right)^ke^{-\frac{3t^2}{\pi^2}}.
    \end{equation}
    Summing over all $k\geq 1$ we have 
    \begin{equation}
        \lim_{n \to \infty} \textnormal{Leb}\bigg(\bigg\{\alpha \in \T: \#\{T \in \Delta_n: a_Tu(\alpha)\Gamma' \in A_n\}>0 \bigg\}\bigg)=1-e^{-\frac{3t^2}{\pi^2}}.
    \end{equation}
Note that for large enough $n$, the set $$\bigg\{\alpha \in \T: \#\{T\in \Delta_n:a_Tu(\alpha)\Gamma' \in A_n\}>0\bigg\}$$ is the same as the set $$\bigg\{\alpha \in \T: \mathfrak{M}^{(1)}(\Delta_n; \alpha) > 1-\frac{t}{\sqrt{|\Delta_n|\log{(|\Delta_n|)}}}\bigg\}.$$

Thus 
\begin{equation*}
   |\Delta_n|^{1/2}\log(|\Delta_n|)^{1/2}(1-\mathfrak{M}^{(1)}(\Delta_n; \cdot))\xLongrightarrow[\textnormal{Leb}]{} \textnormal{Wei}\left(\frac{3}{\pi^2}, 2\right).
\end{equation*}
\end{proof}

\section{Logarithm laws}\label{sec:loglaw}

In this section, we derive \corref{cor1} and \corref{cor2}. In \cref{loglawsviaevl} we derive the logarithm laws via \thmref{thm:wei} and \thmref{homogeneous}, using arguments similar to \cite[Appendix A]{bjorklund2023poisson}. In \cref{loglawwithoutevl}, we illustrate how one can derive \eqref{cor11} and \eqref{cor12} without using the extreme value laws, just by invoking effective double equidistribution and divergence Borel-Cantelli.
 
\subsection{Logarithm laws from the extreme value laws}\label{loglawsviaevl}

We recall a proposition from \cite[Appendix A]{bjorklund2023poisson} to use it in our context.

Let $(Z, \nu)$ be a probability space with non-negative measurable functions $f_T$ on $Z$, with $T \in [1, \infty)$. Suppose that the minima of the functions have a limiting distribution over sufficiently lacunary subsets $\Delta_n \subset [1, \infty)$ i.e., for every sequence of finite subsets $\Delta_n \subset [1, \infty)$ that satisfies 
\begin{equation} \label{lac}
    |\Delta_n| \to \infty, \frac{{\min}_{T \neq T^{'} \in \Delta_n}\{|\log T -\log T'|\}}{\log|\Delta_n|} \to \infty, \frac{{\min}_{T \in \Delta_n}\{\log T\}}{\log|\Delta_n|} \to \infty \,\, \text{ as } n \to \infty.
\end{equation}
the minima 
\begin{equation*}
    \mathfrak{m}_{\Delta_n}:=\min\{f_T: T \in \Delta_n\}
\end{equation*}
satisfy 
\begin{equation*}
    |\Delta_n|^a(\log|\Delta_n|)^b \mathfrak{m}_{\Delta_n} \xLongrightarrow[\nu]{} \textnormal{Wei}(c,d)
\end{equation*}
for some $a, c, d>0$ and  $b\geq0$.

We recall \cite[Proposition A.2]{bjorklund2023poisson} from Appendix A of \cite{bjorklund2023poisson}.
\begin{lem}\label{loglawlemma}
    Under the above assumptions, for every $\delta<1$ for $\nu$ almost every $z \in Z$ one has, $$\liminf_{T\to \infty} (\log(T))^{a\delta}f_{T}(z)=0.$$
\end{lem}

\begin{proof}[Proof of \corref{cor1}]
    By \thmref{thm:wei} one has that the minima of the function $k^{(1)}_{\xi}(T;\alpha)$ satisfy 
    \begin{equation*}
   |\Delta_n|\mathfrak{m}^{(1)}_\xi(\Delta_n; \cdot)\xLongrightarrow[\textnormal{Leb}]{} \textnormal{Wei}(2, 1)
\end{equation*}
for sufficiently lacunary subsets $\Delta_n \subset [1, \infty).$ Thus by \lemref{loglawlemma} we have for all $\delta<1$, and almost every $\alpha \in \T$  $$\liminf_{T \to \infty} \log(T)^{\delta}k^{(1)}_\xi(T;\alpha)=0.$$

Thus we have for almost every $\alpha \in \T$ $$\limsup_{T \to \infty} \frac{-\log(k^{(1)}_\xi(T;\alpha))}{\log(\log(T))} \geq 1.$$

To prove the reverse inequality, observe that any $T$ large enough can be sandwiched between two powers of $2$, i.e., there exists $n \in \N$ such that $2^{n} \leq T < 2^{n+1}$. So, one has 
\begin{equation}\label{sandwich}
    2k_{\xi}^{(1)}(2^n;\alpha)\geq Td^{(1)}_{\xi}(2^n;\alpha) \geq k_{\xi}^{(1)}(T;\alpha) \geq T d^{(1)}_{\xi}(2^{n+1};\alpha) \geq \frac{1}{2} k_{\xi}^{(1)}(2^{n+1};\alpha)
\end{equation}

So it is enough to show the reverse inequality along the sequence $T=2^n$.

\begin{claim}
For any $\varepsilon>0$ and $C>0$, the inequality $$k_{\xi}^{(1)}(2^{n};\alpha) \leq \frac{C}{n^{1+\varepsilon}}$$
holds only for finitely many $n$ for almost every $\alpha \in \T$.
\end{claim}
\begin{proof}[Proof of claim]
    Note that for large enough $n$, one has $k^{(1)}_{\xi}(2^n;\alpha)\leq\frac{C}{n^{1+\varepsilon}}$ if and only if $a_{2^n}u(\alpha)[1_2, -\bm{\xi}]\Gamma \in A^{(1)}_{\frac{C}{n^{1+\varepsilon}}}$. Let $h_n^{+}$ be a smooth approximation of the set $A^{(1)}_{\frac{C}{n^{1+\varepsilon}}}$ from above (using eventual $K$-smoothness of $A^{(1)}_{\frac{C}{n^{1+\varepsilon}}}$), by taking $\delta=\frac{1}{n^2}$ in \lemref{lemma:2}

    \begin{align*}
\text{Leb}\left(\left\{\alpha \in \T: k_{\xi}^{(1)}(2^{n};\alpha) \leq \frac{C}{n^{1+\varepsilon}} \right\}\right) & \leq \int_{0}^{1} h_n^+(a_{2^n} u(x) [1_2, -\boldsymbol{\xi}] \Gamma) \,  dx \notag \\
& \leq \int_{G/\Gamma} h_n^+ + {C} \|h_n^+\|_{C_b^8} 2^{-\delta_1 n} \notag \\
& \leq \mu(A_{\frac{C}{n^{1+\varepsilon}}}^{(1)}) + \frac{K}{n^2} + \frac{\tilde{C} n^{M}}{2^{\delta_1 n}} \notag\\
& \leq \frac{c}{n^{1+\varepsilon}} + \frac{K}{n^2} + \frac{\tilde{C} n^M}{2^{\delta_1 n}}
\end{align*}
Thus we have $$\sum_{n} \textnormal{Leb} \left(\left\{\alpha \in \T: k_{\xi}^{(1)}(2^{n};\alpha) \leq \frac{C}{n^{1+\varepsilon}} \right\}\right)<\infty.$$ So by the convergence case of the Borel-Cantelli Lemma we obtain that for almost every $\alpha \in \T$, the inequality $k_{\xi}^{(1)}(2^{n};\alpha) \leq \frac{C}{n^{1+\varepsilon}}$
holds for only finitely many $n$.
\end{proof}
Thus we have for all $\varepsilon>0$, almost every $\alpha \in \T$  $$1+\varepsilon\geq\limsup_{n \to \infty}\frac{-\log(k^{(1)}_\xi(2^n;\alpha))}{\log(n)}=\limsup_{T \to \infty} \frac{-\log(k^{(1)}_{\xi}(T;\alpha))}{\log(\log(T))}.$$
This proves that for almost every $\alpha \in \T$
$$\limsup_{T \to \infty} \frac{-\log(k^{(1)}_\xi(T;\alpha))}{\log(\log(T))} = 1.$$

One can run a similar argument for $k^{(r)}_\xi(T;\alpha)$ for $r\geq 2$. 

\noindent To obtain $$\limsup_{T \to \infty} \frac{-\log(k^{(r)}_\xi(T;\alpha))}{\log(\log(T))} \geq \frac{1}{2}$$
one can again use \lemref{loglawlemma}. To prove the $\leq \frac{1}{2}$ statement one can give an analogous convergence case Borel-Cantelli lemma argument to conclude (noting that $\mu(A^{(r)}_t)\ll t^2$ for all $r\geq 2$ and $t$ small).
 
Now we prove \eqref{cor13}. We give an argument analogous to the proof of \lemref{loglawlemma} to get the non-trivial side of the logarithm law.

\begin{claim}For all $\delta>0$ and any $C>0$, one has \begin{equation*}
    \textnormal{Leb}\left(\left\{ \alpha \in \T: k^{(1)}_\xi(T;\alpha) \leq C(\log T)^{1-\delta} \;\text{for all large enough} \; T\geq T_{0}(\alpha) \right\}\right)=0
\end{equation*}
\end{claim}
\begin{proof}[Proof of claim] Assume the contrary for the sake of contradiction, i.e. there exists $\delta >0$ and $\varepsilon>0$ such that 
    \begin{equation}\label{assumption}
    \text{Leb}\left(\left\{ \alpha \in \T: k^{(1)}_\xi(T;\alpha) \leq C(\log T)^{1-\delta} \;\text{for all large enough} \; T\geq T_{0}(\alpha) \right\}\right)>\varepsilon.
\end{equation}
Then for a sequence of $\Delta_n$ satisfying \eqref{lac} one has 
\begin{multline}\label{contain}
    \left\{ \alpha \in \T: k^{(1)}_\xi(T;\alpha) \leq C(\log T)^{1-\delta} \;\text{for all large enough} \; T\geq T_{0}(\alpha) \right\} \subset \\ \left\{ \alpha \in \T: \mathfrak{M}_{\xi}^{(1)}(\Delta_n; \alpha) \leq C \max_{T\in \Delta_n}\{\log(T)\}^{1-\delta} \;\text{for all large enough} \; n \right\}. 
\end{multline}
Define $\Omega_n:=\{ \alpha \in \T: \mathfrak{M}_{\xi}^{(1)}(\Delta_n;\alpha) \leq C \max_{T \in \Delta_n}\{\log(T)\}^{1-\delta}\}$. Then \eqref{contain} and our assumption \eqref{assumption} gives $\text{Leb}(\liminf \Omega_n)>\varepsilon$, which implies $\liminf \text{Leb}(\Omega_n)>\varepsilon$ (Fatou's lemma). Therefore one has for all large enough $n$, $\text{Leb}(\Omega_n)>\varepsilon$. Let $u>0$ be chosen to be small enough such that $e^{-\frac{1}{u\pi^2}}<\varepsilon/2$. Now if we choose a sequence of $\Delta_n \subset [1, \infty)$ such that \eqref{lac} is satisfied and such that 
\begin{equation}\label{lacc}
\frac{|\Delta_n|}{\max_{T \in \Delta_n}\{\log(T)\}^{1-\delta}} \to \infty,\end{equation}
then for large enough $n$ one has that $$\Omega_n \subset \{\alpha\in \T: \mathfrak{M}^{(1)}_\xi(\Delta_n;\alpha) \leq u |\Delta_n| \}.$$

However, $\text{Leb}(\{\alpha\in \T: \mathfrak{M}^{(1)}_\xi(\Delta_n;\alpha) \leq u |\Delta_n| \}) \to e^{-\frac{1}{u\pi^2}} <\varepsilon/2.$ We get a contradiction provided we show the existence of a sequence of $\Delta_n \subset [1, \infty)$ satisfying \eqref{lac}
and \eqref{lacc}. It is easy to verify that the choice of $\Delta_n:=\{\theta_n^{s}\;:1\leq s\leq l_n\}$, where we can set $\theta_n=l_n^{\log(l_n)}$ for a sequence $l_n \to \infty$, satisfies \eqref{lac} and \eqref{lacc}. \end{proof}

From the claim, we conclude for almost every $\alpha \in \T$ and for every $\delta>0$ we have that the inequality $$k^{(1)}_\xi(T;\alpha) >C(\log T)^{1-\delta}$$ has infinitely many solutions. Taking logarithms we have for almost every $\alpha \in \T$ and every $\delta>0$ $$\limsup_{T \to \infty}\frac{\log(k^{(1)}_\xi(T;\alpha))}{\log(\log(T))}\geq 1-\delta.$$

Proving that for almost every $\alpha \in \T$ one has 
$$\limsup_{T \to \infty}\frac{\log(k^{(1)}_\xi(T;\alpha))}{\log(\log(T))}\geq 1.$$ For the $\leq 1$ side one can give an analogous convergence case Borel-Cantelli argument as in the claim of the proof of \corref{cor1}. 
\end{proof}

\begin{proof}[Proof of \corref{cor2}]
        By \thmref{homogeneous} one has that the minima of the functions $k_{0}^{(r)}(T;\alpha)$ satisfy
        \begin{equation*}
   |\Delta_n|\mathfrak{m}^{(r)}(\Delta_n; \cdot)\xLongrightarrow[\textnormal{Leb}]{} \textnormal{Wei}\left(\frac{12}{\pi^2r^2}, 1\right) \; \text{for all} \; r\geq 1 
\end{equation*}
for sufficiently lacunary subsets $\Delta_n \subset [1, \infty)$. Thus, by \lemref{loglawlemma}, we have for almost every $\alpha \in \T$ 
$$\limsup_{T \to \infty} \frac{-\log(k^{(r)}_0(T;\alpha))}{\log(\log(T))} \geq 1.$$
By an argument analogous to the claim in the proof of \corref{cor1} (noting that $\mu'(B^{(r)}_t) \ll t$), one has that for every $r\geq 1$, almost every $\alpha \in \T$
$$\limsup_{T \to \infty} \frac{-\log(k^{(r)}_0(T;\alpha))}{\log(\log(T))} \leq 1.$$
To prove \eqref{cor22}, note that the random variable $1-\mathfrak{M}^{(1)}(\Delta_n; \alpha)$ is the same as the random variable $\min_{T\in \Delta_n}\{1-k^{(1)}_0(T;\alpha)\}$, thus by \thmref{homogeneous} and \lemref{loglawlemma} we have for almost every $\alpha \in \T$
$$\limsup_{T \to \infty} \frac{-\log(1-k^{(1)}_0(T;\alpha))}{\log(\log(T))} \geq \frac{1}{2}.$$
\end{proof}

\begin{remark}\label{truelimit}
    In fact, the limsup $$\limsup_{T \to \infty} \frac{-\log(1-k^{(1)}_0(T;\alpha))}{\log(\log(T))}$$ can be seen to be 1 for almost every $\alpha \in \T$ by a direct application of \cite[Theorem 1.8]{kleinbock2018zero} for the function $\psi(T)=\frac{1-(\log T)^{-\kappa}}{T}$. Since the relevant sum $$\sum_{n \geq 3} \frac{\kappa \log(\log n)}{n(\log n)^\kappa}$$ converges for $\kappa>1$, and diverges for $\kappa \leq 1$.
\end{remark}

\subsection{Logarithm laws without using the extreme value laws}\label{loglawwithoutevl}

We introduce the following notation just for this subsection. Let $\theta>0$ be fixed and let $$A_n^{(1)}:=A^{(1)}_{\frac{1}{n^\theta}}, \quad A^{(2)}_n:=A^{(2)}_{\frac{1}{n^\theta}}.$$ Here $A^{(1)}_{\frac{1}{n^\theta}}$ and $A^{(2)}_{\frac{1}{n^\theta}}$ on the right hand side are the regions of $G/\Gamma$ as defined in \eqref{eucregion} and \eqref{regions}. 
The following is the main dynamical result of this subsection.

\begin{thm}\label{thm:main_dyn}
For $\theta \leq 1$ and $\xi \in \mathbb{T}$  of Diophantine type, we have for almost all $\alpha \in \mathbb{T}$ \begin{equation} a_{2^n} u(\alpha) \left[\mathrm{1}_2, \mt{\xi}0 \right] \Gamma \in A_n^{(1)} \label{eq:1} \end{equation} for infinitely many $n$. For $\theta>1$, we have for almost all $\alpha \in \mathbb{T}$ \eqref{eq:1} holds only for finitely many $n$.\\

Similarly, for $\theta \leq \frac{1}{2}$ and $\xi \in \mathbb{T}$ of Diophantine type we have for almost all $\alpha \in \mathbb{T}$ \begin{equation} a_{2^n} u(\alpha) \left[\mathrm{1}_2, \mt{\xi}0 \right] \Gamma \in A_n^{(2)} \label{eq:2} \end{equation} for infinitely many $n$. For $\theta >\frac{1}{2}$, for almost all $\alpha \in \mathbb{T}$ \eqref{eq:2} only holds for finitely many $n$. 
\end{thm}
\begin{proof}[Proof of \eqref{cor11} and \eqref{cor12} assuming Theorem \ref{thm:main_dyn}]

Note that if $\xi$ is of Diophantine type, then so is $-\xi$. The affine lattice corresponding to the element $a_{2^n} u(\alpha) \left[\mathrm{1}_2, -\mt{\xi}0 \right] \Gamma$ is $$\Ll=\left\{\mt{2^np+2^n\alpha q-2^n\xi}{q2^{-n}}: p, q \in \Z \right\}.$$
\thmref{thm:main_dyn} implies for infinitely many $n$ it has  at least one point in the region $E_\frac{1}{n}$ which implies that there exist infinitely many $n$ such that $d_{\xi}^{(1)}(2^n; \alpha)\leq\frac{1}{n2^n}$. By taking logarithms and rearranging this implies that for almost all $\alpha$ $$\limsup_{n \to \infty} \frac{-\log(d_{\xi}^{(1)}(2^n; \alpha))-n\log2}{\log(n)} \geq 1.$$

Since for $\theta > 1$ by \thmref{thm:main_dyn}, for almost all $\alpha$ there are only finitely many $n$ such that $d_{\xi}^{(1)}(2^n;\alpha)\leq\frac{1}{n^{\theta}2^n}$, one concludes  $$\limsup_{n\to \infty} \frac{-\log(d_{\xi}^{(1)}(2^n;\alpha))-n\log2}{\log(n)} = 1.$$
This gives, by a sandwiching argument similar to \eqref{sandwich}, for almost every $\alpha \in \T$ $$\limsup_{T\to \infty} \frac{-\log(k^{(1)}_\xi(T;\alpha))}{\log(\log(T))}=1.$$
Similarly, for $r=2$ we have that the affine lattice $\Ll$ has two points in $E_\frac{1}{\sqrt{n}}$ for infinitely many $n$ implies that for almost all $\alpha$ there exists $0\leq q_1< q_2 \leq 2^n$ such that $d(q_1\alpha, \xi) \leq \frac{1}{n^{1/2}2^n}$ and  $d(q_2\alpha, \xi) \leq \frac{1}{n^{1/2}2^n}$. This is equivalent to saying that for almost all $\alpha$ there exist infinitely many $n$ such that $d_{\xi}^{(2)}(2^n;\alpha) \leq \frac{1}{n^{1/2}2^n}$, which by similar rearrangement and argument as above gives $$\limsup_{T \to \infty} \dfrac{-\log(k_\xi^{(2)}(T;\alpha))}{\log(\log(T))}=\frac{1}{2}.$$
For $r\geq3$, note that $d^{(r)}_\xi(T; \alpha) \geq d^{(2)}_\xi(T; \alpha)$. So, 
$$\limsup_{T \to \infty} \dfrac{-\log(k_\xi^{(r)}(T;\alpha))}{\log(\log(T))}\leq\frac{1}{2}.$$

Now to prove the reverse inequality notice that if $d^{(2)}_\xi(n;\alpha)=\varepsilon$, then there exists $0\leq q_1<q_2\leq n$ such that $d(q_1\alpha, \xi) \leq \varepsilon$ and $d(q_2\alpha, \xi) \leq \varepsilon$. Set $k=q_2-q_1$, then for $s=0, \ldots, r-1$, we have $d(q_1\alpha+sk\alpha, \xi) \leq (1+2s)\varepsilon$. Thus we have $d_{\xi}^{(r)}(rn;\alpha)\leq (2r-1)d_\xi^{(2)}(n;\alpha)$. Taking the limsup, we get for $r\geq3$
$$\limsup_{T \to \infty} \dfrac{-\log(k_\xi^{(r)}(T;\alpha))}{\log(\log(T))}=\frac{1}{2}.$$
\end{proof}

Let us now prove \thmref{thm:main_dyn}. Recall the statement of the divergence Borel-Cantelli lemma:

\begin{lem}\label{DBC}
Let $(X, \mathcal{A}, \mu)$ be a probability space. Let $\{E_i\}_{i=1}^{\infty}$ be a sequence of events in $X$ such that:
\begin{enumerate}
    \item $\sum_{i=1}^{\infty} \mu(E_i) = \infty$
    \item There exists a constant $C > 0$ such that 
    \[ \sum_{s,t=1}^{Q} \mu(E_s \cap E_t) \leq C \left( \sum_{s=1}^{Q} \mu(E_s) \right)^2 \] 
    for infinitely many $Q \in \mathbb{N}$.
\end{enumerate}
Then, $\mu(\limsup_{n} E_n) \geq \frac{1}{C}$.
\end{lem}
\noindent See \cite[Lemma DBC]{BERESNEVICH2023126750} and the references listed therein for history and proof of \lemref{DBC}. \\

Now, for such a fixed $\theta$ and $\xi$ define the sets \begin{equation*} C_n := \left\{ x \in \mathbb{T} : a_{2^n} u(x) [1_2, \boldsymbol{\xi}] \Gamma \in A_n^{(1)} \right\} \end{equation*}
\begin{equation*}
    B_n := \left\{ x \in \mathbb{T} : a_{2^n} u(x) [1_2, \boldsymbol{\xi}] \Gamma \in A_n^{(2)} \right\}.
\end{equation*}
Also let $C_\infty=\limsup C_n$ and $B_\infty=\limsup B_n.$
    \begin{definition}
Given a subinterval $[\alpha, \beta] \subseteq [0, 1]$ with $\alpha \neq \beta$, let $\ell_{[\alpha, \beta]}$ be the normalised Lebesgue measure on the interval $[\alpha, \beta]$ defined as $\ell_{[\alpha, \beta]}(A):=\frac{\textnormal{Leb}(A \cap [\alpha, \beta])}{\beta - \alpha}$ for all measurable subsets $A \subseteq [0,1]$.
\end{definition}

\begin{lem}\label{int}
 If $\theta \leq1$, there exists an $\eta>0$ such that for all non-degenerate intervals $[\alpha, \beta] \subseteq [0, 1]$ we have $$\ell_{[\alpha, \beta]}(C_{\infty}) = \ell_{[\alpha, \beta]}(\limsup_n C_n) \geq \eta.$$

Similarly if $\theta \leq \frac{1}{2}$, there exists an $\eta'>0$ such that for all non-degenerate intervals $[\alpha, \beta] \subseteq [0,1]$ we have $$\ell_{[\alpha, \beta]}(B_{\infty}) = \ell_{[\alpha, \beta]}(\limsup_n B_n) \geq \eta'.$$

\end{lem}

\begin{proof}
Note that the sets $A_n^{(1)}$ are $K'$-smooth for some $K'$. So by \lemref{smooth}, there exists a $K>0$ such that we can choose smooth functions $h_n^{\pm}$ such that 

\begin{equation}\label{smoothappro1}
    0\leq h_n^{-}\leq\chi_{A_n}\leq h_n^{+}\leq 1, \quad \|\chi_{A_n}-h_n^{\pm}\|_{L^1}\leq\frac{1}{n^2}, \quad \|h_{n}^{\pm}\|_{\mathcal{C}^k}\leq C' n^{K}.
\end{equation}
Note that \thmref{stt} gives us for all $n \in \mathbb{N}$ 
\begin{align}
\ell_{[\alpha, \beta]}(C_n) & \geq \frac{1}{\beta - \alpha} \int_{\alpha}^{\beta} h_n^-(a_{2^n} u(x) [1_2, \boldsymbol{\xi}] \Gamma) \,  dx \notag \\
& \geq \int_{G/\Gamma} h_n^- - \frac{C}{\beta - \alpha} \|h_n^-\|_{C_b^8} 2^{-\delta_1 n} \notag \\
& \geq \mu(A_n^{(1)}) - \frac{1}{n^2} - \frac{\tilde{C} n^K}{(\beta - \alpha) 2^{\delta_1 n}}. 
\label{1}
\end{align}

In particular, since for $\theta \leq 1$ the sum $\sum_{i=1}^\infty \mu(A_{i}^{(1)})$ diverges we have that $\sum_{i=1}^{\infty} \ell_{[\alpha, \beta]} (C_i)$ diverges for all $[\alpha, \beta] \subseteq [0,1]$.\\
Without loss of generality, we assume that $s\geq m > n \geq 1$. For an interval $[\alpha, \beta] \subset [0,1]$ take a smooth function $\varphi_\delta$ as in Lemma \ref{apr}, with $\delta=1/s^2$. Then using \thmref{multiple} for $r=2$ and the properties  \eqref{smoothappro1} of $h_n^{+}, h_m^{+}$ and $\varphi_\delta$ we have
\begin{align}
\ell_{[\alpha, \beta]}(C_n \cap C_m) & \leq \frac{1}{\beta - \alpha} \int_{0}^{1} \varphi_\delta(x) h_n^+(a_{2^n} u(x)[1_2, \boldsymbol{\xi}] \Gamma) \, h_m^+(a_{2^m} u(x) [1_2, \boldsymbol{\xi}] \Gamma) \, dx \notag \\
& \leq \frac{1}{\beta-\alpha} \int_{0}^1 \varphi_\delta \cdot\int_{G/\Gamma} h_n^+ \cdot \int_{G/\Gamma} h_m^+ \notag \\
& \hspace{4cm}+ \frac{D\|\varphi_\delta\|_{W}}{(\beta - \alpha)} \|h_n^+\|_{\Cc^k} \|h_m^+\|_{\Cc^k} \, 2^{-\delta_2 \min \{m, \; n,\; m-n\}} \notag \\
&\leq \left(1+\frac{4}{s^2(\beta-\alpha)}\right) \left( \mu(A_n^{(1)}) + \frac{1}{n^2} \right) \left( \mu(A_m^{(1)}) + \frac{1}{m^2} \right) \notag \\
& \hspace{6 cm} + \frac{C_1}{\beta - \alpha} s^2n^K m^K 2^{ - \delta_2 \min \{n, m-n\}}.
\label{2}
\end{align}
So from \eqref{1} and \eqref{2} we have

\begin{align}
\ell_{[\alpha,\beta]}(C_n\cap C_m)
&\leq
\left(1+\frac{4}{s^2(\beta-\alpha)}\right)
\left(
  \ell_{[\alpha,\beta]}(C_n)
  +\frac{2}{n^2}
  +\frac{\tilde C}{\beta-\alpha}
   \frac{n^K}{2^{\delta_1 n}}
\right)
\nonumber\\
&\hspace{3.5cm}\times
\left(
  \ell_{[\alpha,\beta]}(C_m)
  +\frac{2}{m^2}
  +\frac{\tilde C}{\beta-\alpha}
   \frac{m^K}{2^{\delta_1 m}}
\right)
\nonumber\\
& \hspace{6cm}
+\frac{C_1}{\beta-\alpha}
 s^2n^Km^K
 2^{-\delta_2\min\{n,m-n\}}.
\end{align}
Set for $j \in \N$, $b_j = \frac{1}{\beta-\alpha}\left[ \frac{2}{j^2} + \frac{j^K \tilde{C}}{2^{\delta_1 j}} \right].$ This gives for $s$ large enough 

\begin{eqnarray}
\ell_{[\alpha, \beta]}(C_n \cap C_m) - \ell_{[\alpha, \beta]}(C_n) \ell_{[\alpha, \beta]}(C_m) &\leq& 2\ell_{[\alpha, \beta]}(C_n) b_m + 2\ell_{[\alpha, \beta]}(C_m) b_n +2b_nb_m\\ \notag
&& +\frac{C_1}{\beta - \alpha}s^{2K+2}2^{-\delta_2 \min \{n, m-n\}}+\frac{4}{s^2(\beta-\alpha)} \notag
\end{eqnarray}

{\textbf{\underline{Case 1:}} When $n > \ceil*{\frac{(2K+4)}{\delta_2} \log_2 s}$ and $m-n > \ceil*{\frac{(2K+4)}{\delta_2} \log_2 s}$:}\\

Then,
\[\ell_{[\alpha, \beta]}(C_n \cap C_m) - \ell_{[\alpha, \beta]}(C_n) \ell_{[\alpha, \beta]}(C_m) \leq 2\ell_{[\alpha, \beta]}(C_n) b_m + 2\ell_{[\alpha, \beta]}(C_m) b_n + 2b_nb_m+ \frac{(C_1+4)}{(\beta - \alpha)s^2}.
\]
So,
\begin{align*}
\sum_{\substack{s \geq m > n \geq 1 \\ m-n,\; n > \ceil*{\frac{(2K+4)}{\delta_2}{\log_2 s}}}} \ell_{[\alpha, \beta]}(C_n \cap C_m) - \ell_{[\alpha, \beta]}(C_n) \ell_{[\alpha, \beta]}(C_m)
&\leq 4 \left( \sum_{i=1}^{s} b_i \right) \left( \sum_{i=1}^{s} \ell_{[\alpha, \beta]}(C_i) \right)\\
& \quad + \dfrac{(C_1+4)}{\beta - \alpha}+ \left(\sum_{i=1}^s b_i\right)^2.
\end{align*}
Since $\sum_{i=1}^{\infty} \ell_{[\alpha, \beta]}(C_i) = \infty$ there exists $s_0$ (depending on $[\alpha, \beta]$) such that $\forall \, s \geq s_0$

\[
\sum_{i=1}^{s} \ell_{[\alpha, \beta]}(C_i) \geq \frac{1}{\beta - \alpha} \geq 1.
\]
Noting that $\sum_{j \in \N} b_j=\frac{c}{\beta-\alpha}$ for some $c>0$ independent of $[\alpha, \beta]$, one has for all large enough $s$
\[
\sum_{\substack{s \geq m > n \geq 1 \\ m-n, \; n > \ceil*{\frac{(2K+4)}{\delta_2} \log_2 s}} } \ell_{[\alpha, \beta]}(C_n \cap C_m) -\ell_{[\alpha, \beta]}(C_n) \ell_{[\alpha, \beta]}(C_m) \leq K_1 \left[ \sum_{i=1}^{s} \ell_{[\alpha, \beta]}(C_i) \right]^2
\] where $K_1$ is independent of $\alpha$ and $\beta$.\\

\underline{\textbf{Case 2:}} When $1 \leq n \leq \ceil*{\left( \frac{2K+4}{\delta_2} \right) \log_2 s}$:\\

Then using the trivial bound $\ell_{[\alpha, \beta]}(C_n \cap C_m) -\ell_{[\alpha, \beta]}(C_n)\ell_{[\alpha, \beta]}(C_m)\leq \ell_{[\alpha, \beta]}(C_m)$, we obtain
\[
\sum_{\substack{s \geq m > n \geq 1 \\ 1 \leq n \leq \ceil*{\left( \frac{2K+4}{\delta_2} \right)\log_2 s}}} \ell_{[\alpha, \beta]}(C_n \cap C_m) - \ell_{[\alpha, \beta]}(C_n) \ell_{[\alpha, \beta]}(C_m) \leq \left( \frac{2K+4}{\delta_2} \right) \log_2 s \left[ \sum_{i=1}^{s} \ell_{[\alpha, \beta]}(C_i) \right].
\]\\

\underline{\textbf{Case 3:}} When $1 \leq m-n \leq \ceil*{\left( \frac{2K+4}{\delta_2} \right) \log_2 s}$ and $n>  \ceil*{\left( \frac{2K+4}{\delta_2} \right) \log_2s}$:\\

Similarly using the trivial bound we have
\begin{align*}
\sum_{\substack{s \geq m > n \geq 1 \\ 1 \leq m-n \leq \ceil*{\left( \frac{2K+4}{\delta_2} \right) \log_2 s},  \; n> \ceil*{\left( \frac{2K+4}{\delta_2} \right) \log_2 s}}} \ell_{[\alpha, \beta]}(C_n \cap C_m) - \ell_{[\alpha, \beta]}(C_n) \ell_{[\alpha, \beta]}(C_m) & \\ \leq \left( \frac{2K+4}{\delta_2} \right) \log_2 s \left( \sum_{i=1}^{s} \ell_{[\alpha, \beta]}(C_i) \right).
\end{align*}
Note that for $\theta \leq 1$ we have
$\ell_{[\alpha, \beta]}(C_n) \geq \mu(A_n^{(1)}) - \frac{1}{n^2} - \frac{\tilde{C} n^K}{(\beta - \alpha) 2^{\delta_1 n}}$ giving us,

$$ \sum_{i=1}^{s} \ell_{[\alpha, \beta]}(C_i) \geq \sum_{i=1}^{s} \mu(A_i^{(1)}) - C_{[\alpha, \beta]} \geq \sum_{i=1}^{s} \frac{\tilde{C}}{i} - C_{[\alpha, \beta]} \geq A \log_2 s $$ for some  $A >0$ independent of $[\alpha, \beta] \subseteq$ [0, 1] for all large enough $s$.  

Now combining cases 1, 2, and 3 we have for all $[\alpha, \beta] \subseteq [0,1]$ for all sufficiently large $s$, $$\sum_{s \geq m>n\geq 1} \ell_{[\alpha, \beta]}(C_n \cap C_m) - \ell_{[\alpha, \beta]}(C_n) \ell_{[\alpha, \beta]}(C_m) \leq K' \left(\sum_{i=1}^s \ell_{[\alpha, \beta]}(C_i)\right)^2$$ where $K'$ is independent of $[\alpha, \beta]$. By symmetry this also controls the $n>m$ terms, and the diagonal contributes at most $\sum_{i \leq s} \ell_{[\alpha, \beta]}(C_i)$, which can be absorbed by the square term for large $s$. Now, using \lemref{DBC} we conclude. A similar argument also goes through for $B_\infty$ when $\theta \leq \frac{1}{2}$, since $\mu(A^{(2)}_n) \asymp n^{-2\theta}$.
\end{proof}

Recall the following lemma (see \cite[Lemma 1.6]{harman1998metric} for a proof), which is a consequence of Lebesgue's density theorem:

\begin{lem}\label{Leb}
    Let $A$ be a measurable subset of $\R$ such that there exists a (uniform) $\delta>0$ such that $$\textnormal{Leb}(A \cap I)\geq\delta  \; \textnormal{Leb(I)}$$
    for every finite subinterval $I \subset \R$. Then, $A$ is a set of full measure.
\end{lem}

\begin{proof}[Proof of Theorem \ref{thm:main_dyn}]\underline{Divergence case:}\\
Note that $C_\infty$ and $B_\infty$ are   1-periodic subsets of $\R$. By Lemma \ref{int}, we obtain that for $\theta \leq 1$ for any interval $I$ of $\R$ we have $\textnormal{Leb}(C_\infty \cap I) \geq \eta \; \textnormal{Leb}(I)$. Thus, Lemma \ref{Leb} implies that $C_\infty$ is a set of full measure for $\theta \leq 1$. A similar argument shows that the set $B_\infty$ is a set of full measure for $\theta \leq \frac{1}{2}$.

\noindent \underline{Convergence case:}\\
We have by \thmref{stt}

\begin{align}
\ell_{[0,1]}(C_n) & \leq \int_{0}^{1} h_n^+(a_{2^n} u(x) [1_2, \boldsymbol{\xi}] \Gamma) \,  dx \notag \\
& \leq \int_{G/\Gamma} h_n^+ + {C} \|h_n^+\|_{C_b^8} 2^{-\delta_1 n} \notag \\
& \leq \mu(A_n^{(1)}) + \frac{1}{n^2} + \frac{\tilde{C} n^K}{2^{\delta_1 n}} \notag\\
& \leq \frac{c}{n^{\theta}} + \frac{1}{n^2} + \frac{\tilde{C} n^K}{2^{\delta_1 n}} \nonumber
\end{align}
Thus when $\theta>1$, we have $\sum_{n} \ell_{[0,1]}(C_n)<\infty$. So by the convergence case of Borel-Cantelli Lemma we obtain that $\textrm{Leb}(C_\infty)=0$. A similar argument for $B_\infty$ shows that if $\theta>\frac{1}{2}$ we have $\textrm{Leb}(B_\infty)=0$.

\end{proof}
\nocite{*}
\printbibliography

\end{document}